\documentclass[oneside]{amsart}
\usepackage[dvipdfmx]{graphicx}
\graphicspath{{../img/}{./img/}}
\usepackage{float,amssymb,color,amsmath,amsthm}

\newcommand{\BR}{\mathbb{R}}

\newcommand{\BN}{\mathbb{N}}
\newcommand{\F}{\mathcal{F} } 
\newcommand{\A}{\mathbb{A} } 
\newcommand{\D}{{\bf D}} 
\newcommand{\R}{\mathbb{R} }

\newtheorem{definition}{Definition}
\newtheorem{theorem}{Theorem}
\newtheorem{proposition}{Proposition}
\newtheorem{lemma}{Lemma}
\newtheorem{corollary}{Corollary}
\newtheorem{remark}{Remark}
\newtheorem{example}{Example}

\title[Topological bifurcation structure of one-parameter families]{Topological Bifurcation Structure Of One-Parameter Families Of $C^1$ Unimodal Maps}

\author{Atsuro Sannami$^1$ and Tomoo Yokoyama$^2$ }
\address{$^1$ Kitami Institute of Technology, Kitami 090-8507, Japan
}
\address{$^2$ Department of Mathematics, Kyoto University of Education / JST PRESTO, 
Fujinomori, Fukakusa, Fushimi-ku, Kyoto 612-8522, Japan (Present address: Gifu University, Yanagido 1-1, Gifu, 501-1193, Japan)
}
\email{a.sannami@gmail.com, sannami@mail.kitami-it.ac.jp and tomoo@gifu-u.ac.jp}

\begin{document}

\maketitle

\begin{abstract}
We consider the bifurcation structure of one-parameter families of unimodal maps whose differentiability is only $C^1$. The structure of its bifurcation diagram can be a very wild one in such case. However we prove that in a certain topological sense, the structure is the same as that of the standard family of quadratic polynomials. In the case of families of polynomials, irreducible component of the bifurcation diagram can be defined naturally by dividing by the polynomials corresponding to lower periods. We show that such a irreducible component can be defined even if the maps and the family satisfy only a very mild differentiability condition. By removing components of lower periods, the structure of the bifurcation diagram becomes a considerably simplified one. We prove that the symbolic condition for the irreducible components is exactly the same as that of the standard family of quadratic polynomials. When we consider families of maps without strict conditions, the bifurcation diagrams may have infinitely many wild components. We show that such a situation does not affect the irreducible component essentially by proving a separation theorem for compact set in the plane which asserts that a given connected component can be cut out from the compact set by a curve. 

\ 


\end{abstract}

\section{Introduction}

One parameter families of unimodal maps are 
the most widely studied object in the theory of 
non-linear dynamical systems. 
Their bifurcation and the dynamics are beautifully described 
by the kneading theory and the theory of one-dimensional 
dynamical systems \cite{MT}, \cite{G}, \cite{CE}, 
\cite{JR1}, \cite{JR2}, \cite{dMvS}. 
However, most neat results are obtained for \lq\lq  good\rq\rq\ families of    
unimodal maps such as a family of quadratic polynomials, 
a family of maps having a negative Schwartzian derivative or 
a family of maps having only regular bifurcations. 

In many practical problems where a family of unimodal maps appears, 
we might not expect the maps to have such nice properties.
In this paper, we show that even if we do not impose any strict 
conditions on the family of maps, the bifurcation structure 
of one-parameter families of $C^1$ unimodal maps is, 
in a certain topological sense, 
the same as that of the standard family of quadratic polynomials. 

We say that a $C^1$ map $f:\BR\to\BR$ is 
{\em unimodal} if 
$f'(x)>0$ for any $x<0$, $f'(x)<0$ for any $x>0$ and $f'(0)=0$. 
Our definition of unimodal map is so general that 
it allows a map having an infinite number of fixed points 
and intervals of fixed points. 

We are interested in the bifurcation structure of one-parameter family of 
unimodal maps which starts from a map having no non-wandering point 
and ends up with a map with full dynamics. 
By {\em full dynamics}, we mean that the non-wandering set is 
conjugate to the full shift of one-sided sequences of two symbols. 
We shall call such a unimodal map a {\em horseshoe}. 
More precisely,

\begin{definition}
A $C^1$ unimodal map $f$ is a {\em horseshoe} if $f$ satisfies 
the following $(i)$ and $(ii)$ (as shown in Fig. \ref{figure_1}).
\begin{description}
\item[$(i)$] $f$ has two fixed points.
\item[$(ii)$] Let $a$ be the left fixed point. 
There exist $a_1$, $b_1$ and $b$ such that 
\begin{itemize}
\item $a<a_1<b_1<b$, 
\item $f(b)=a$, 
\item $f(a_1)=f(b_1)=b$ and 
\item $|f'(x)|>1 \quad \forall x\in[a,a_1]\cup[b_1,b]\ .$
\end{itemize}
\end{description}
\end{definition}

\begin{figure}
\centerline{\includegraphics[width=5cm]{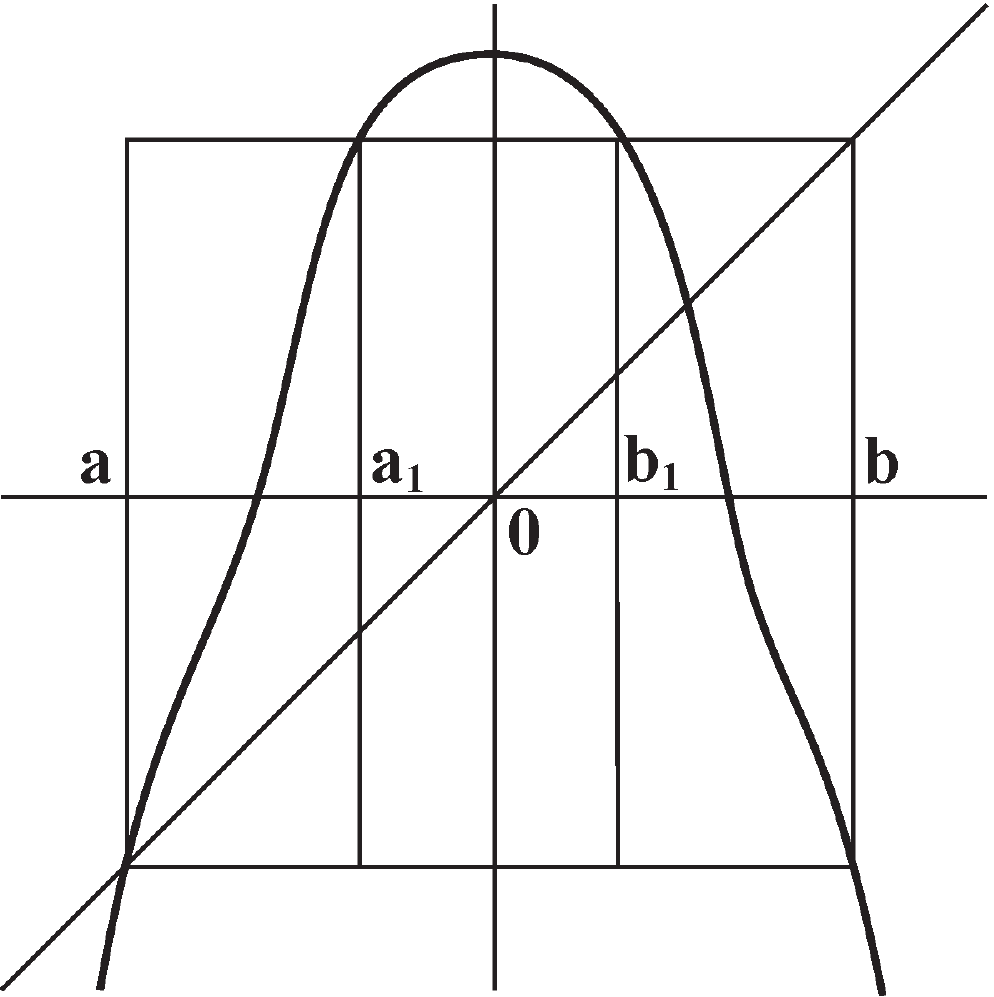}}
\caption{A Horseshoe}
\label{figure_1}
\end{figure}

\begin{remark} 
{\rm 
It is easy to see that $a<a_1<0<b_1<b$. }
\end{remark}

We suppose that the one-parameter families of 
$C^1$ maps we deal with in this paper 
satisfy the following mild continuity and 
differentiability conditions. 

\begin{definition}\label{def:continuous-family}
Let $\{f_t\}_{t\in [0,1]}$ be a one-parameter family of $C^1$ maps 
on $\BR$. We say that $\{f_t\}_{t\in [0,1]}$ is a 
{\em continuous family of $C^1$ maps} if 
$f^0(t,x)=f_t(x)$ and 
$f^1(t,x)={\displaystyle\frac{df_t}{dx}(x)}$  
are continuous functions on $[0,1]\times \BR$. 
\end{definition}

\begin{remark} 
{\rm \item[$(1)$] 
It is easy to see that 
$\{ f_t \}_{t\in [0,1]}$ is a continuous family of $C^1$ maps
if and only if $t \mapsto f_t$ is a continuous curve in the space of 
$C^1$ maps on $\BR$ endowed with the $C^1$ topology.}
\item[$(2)$] {\rm We do not assume that $f^0(t,x)$ is differentiable with respect to $t$. 
Thus, \lq\lq continuous family of $C^1$ maps\rq\rq\  is weaker than 
\lq\lq $f^0(t,x)$ is $C^1$ on $[0,1]\times \BR$~\rq\rq.}
\end{remark}

\begin{definition}\label{def:full-family}
Let $\{f_t\}_{t\in [0,1]}$ be a continuous family of $C^1$ unimodal maps 
on $\BR$. 
We say that $\{f_t\}_{t\in [0,1]}$ is a {\em continuous full family} 
if it satisfies the following properties. 
\begin{description}
\item[$(i)$] $f_0$ has no fixed point (therefore no non-wandering points).
\item[$(ii)$] $f_1$ is a horseshoe. 
\item[$(iii)$] There exists an $M>0$ such that fixed points of $f_t$ are in 
$[-M,M]$ for any $0\le t \le 1$. 
\end{description}
\end{definition}

\begin{remark}\label{rem:unimod}
{\rm 
Since our definition of unimodal map is so general that $f_t$ 
could have infinitely many fixed points diverging to $-\infty$. 
To avoid such a situation, 
we impose the property (iii). 
It is easy to see that no non-wandering point exists 
outside of $[-M,M]$ for any $f_t$. }
\end{remark}

In this paper, we use the following terminology for the 
periodicity of points and sequences. 
Let $f:X\to X$ be a map on a set $X$. 
When $f^n(p)=p$, we say that $p$ is a {\em periodic point of $f$ 
of period $n$}. 
Moreover, if $p$ is not a periodic point of period $k$ for any $0<k<n$,
then we say that {\em the minimal period of $p$ is $n$}. 
The minimal period of $p$ is denoted by $per(p)$. 

Let $\{ f_t \} _{t\in [0,1]}$ be a continuous full family of 
$C^1$ unimodal maps. 
We are interested in the topological structure of the 
bifurcation diagram of periodic points 
of $\{f_t\}$. 
Let $n$ be a positive integer. 
The {\em bifurcation diagram} of periodic points 
of period $n$ is the set 
$\{\,(t,x)\in [0,1]\times \BR\; | \; f_t^n(x)=x \, \}$  
(refer Figure \ref{figure_2}). 
Occasionally we simply call it the bifurcation diagram of period $n$.

\begin{figure}
\centerline{\includegraphics[width=8cm]{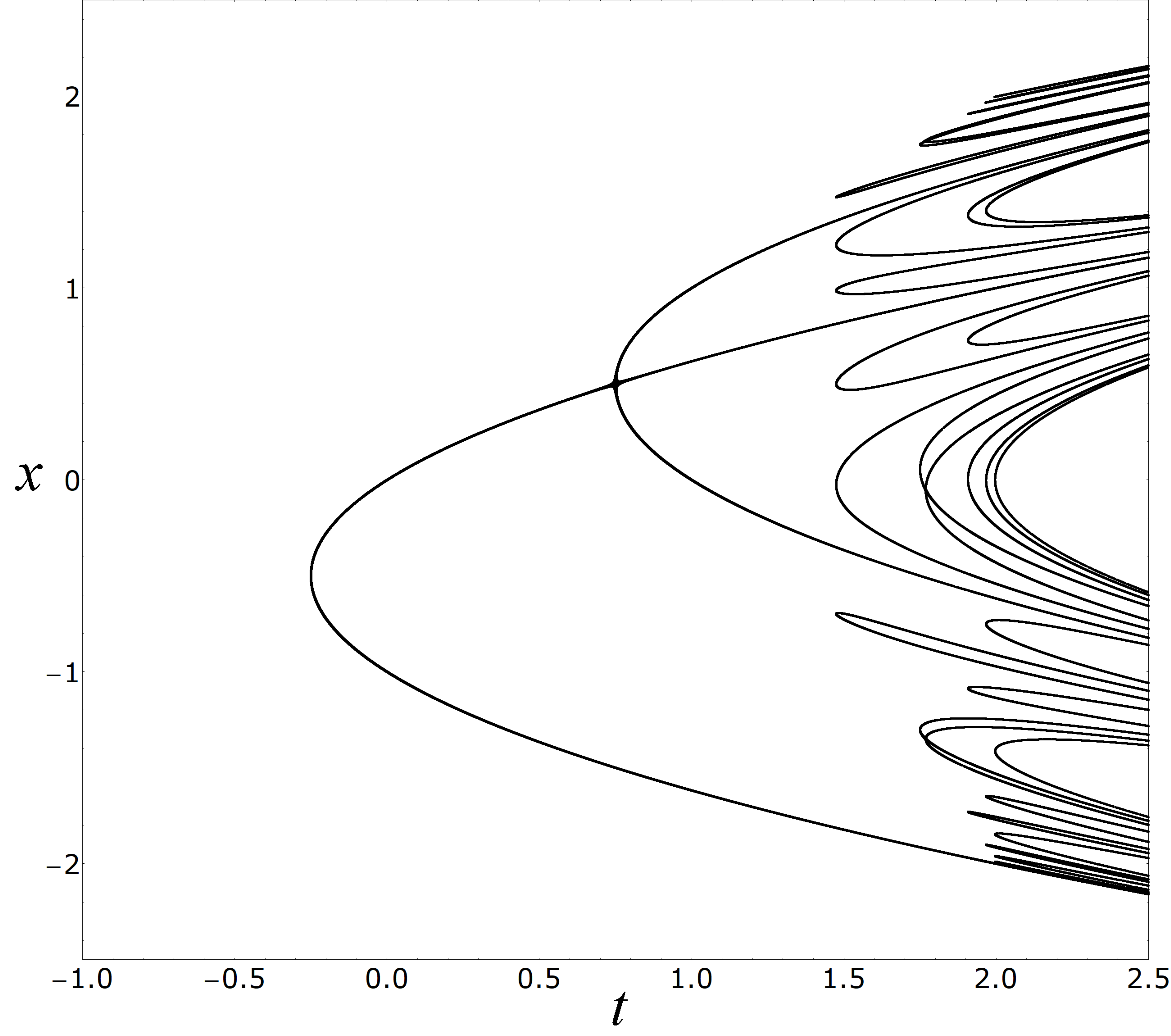}}
\caption{The bifurcation diagram of period 6 of the quadratic family 
$q_t(x)=t-x^2$. 
$f_t(x)=(7/2)t-1-x^2$, a reparametrization of $q_t(x)$, 
satisfies the condition of continuous full family.
By an easy calculation, we can see that 
$q_t$ is a horseshoe for $t>(5+2\sqrt{5})/4 \approx 2.368$. 
}
\label{figure_2}
\end{figure}

By the definition, $f_1$ is a horseshoe. 
Therefore the dynamics of $f_1$ on its non-wandering set 
is equivalent to 
the full shift of one-sided sequences of two symbols, 
and we know rigorously what kind of periodic points $f_1$ has. 

Our question in this paper is; 
\lq\lq What types of periodic points of $f_1$ are 
on the same connected component of the bifurcation diagram ?\rq\rq\   

If there exists even one point where $f_t$ is not differentiable, 
then this question is not interesting at all.  
For example, $f_t(x)=(2t-1)-3|x|$ has only one connected component 
of the bifurcation diagram for any $n\ge 1$ (~$f_t(x)$ is just a 
reparametrization of $g_t(x)=t-3|x|$ ). 

The purpose of this paper is to show that 
for continuous full family of $C^1$ unimodal maps, 
such an equivalence relation of being contained 
in the same connected component of 
the bifurcation diagram is exactly the same as 
that of the standard family of quadratic polynomials. 

We had thought that such a result had been almost obvious or 
it had been known already as a folklore theorem. 
However, we could not find any previous work proving such a statement rigorously. 
And also we found that a considerable amount of non-trivial arguments 
were necessary in the proof.
So we decided to write it down rigorously. 

There are two important points 
on the bifurcation of continuous full family of 
$C^1$ unimodal maps. 
One is the existence of {\it irreducible component} 
of the bifurcation diagram. 
If we consider connected component of the bifurcation diagram itself, 
then some components of different periods have intersection 
because of period doubling bifurcation, 
and the condition of being contained in the same component 
would become a very confusing one. 
To avoid such a confusing situation and make the statement clearer, 
we show that we can divide iterated maps and define irreducible component. 

In the case of polynomial maps, 
irreducible component of the bifurcation diagram can be defined naturally. 
For example, let $g_t(x)=(4t-1)-x^2$ be the standard family of 
quadratic maps and set ${\widetilde G}_2(t,x)=g_t^2(x)-x$ 
(we just reparametrized $q_t(x)=t-x^2$ to get a family consistent with 
the definition of continuous full family). 
Then ${\widetilde G}_2^{-1}(0)$ is a bifurcation diagram of period 2. 
However it also contains a component of fixed points. 
$g_t^2(x)-x$ can be divided by $g_t(x)-x$ and 
$G_2(t,x)={\widetilde G}_2(t,x)/(g_t(x)-x)=x^2-x+2(1-2t)$ 
is also a polynomial of $x$ and $t$. 
$G_2^{-1}(0)$ is the bifurcation diagram of minimal period 2 
except one point of bifurcation, 
and it does not contain the component of fixed points. 
In this paper, we show that such a division is possible, 
and we can define irreducible component of the bifurcation diagram 
even if it satisfies only a very mild differentiability condition 
as in Definition~\ref{def:full-family}. 

Another important point is a separation theorem for compact set in the plane. 
Our family of unimodal maps is so general that its bifurcation diagram may have 
infinitely many wild components. 
We have to guarantee a certain kind of separation of those components 
and show that such a confusing situation does not affect the irreducible 
components topologically. 
We use a recently developed technique \lq\lq filling\rq\rq\,  which was devised 
to analyse homeomorphisms on 2-dimensional manifolds \cite{JKP}, \cite{KT1}, 
\cite{KT2}.

On bifurcation diagram of one-dimensional maps, 
we should mention the works of Guckenheimer, Jonker and Rand. 
A study of the bifurcation diagram of a family of unimodal maps is 
first attempted in a paper of Guckenheimer \cite{G}. 
Assuming certain regularity for maps and families, 
he obtained some qualitative properties of such families. 
\cite{JR2} of Jonker and Rand is a remarkable work 
about a family of $C^1$ unimodal maps. 
Based on the works of 
Jonker \cite{J} and Jonker and Rand \cite{JR1}, 
they investigated family of $C^1$ unimodal maps thoroughly, 
and obtained a universal property of the changing process of kneading invariant 
of such a family. 
Some arguments in the symbolic dynamics part of our paper are 
similar to some arguments in \cite{JR2}, but their result is not about 
the bifurcation diagram itself and our arguments are mainly about 
the symbolic properties of irreducible component which is not defined in \cite{JR2}. 
For this reason, we wrote our paper in a self-contained way.

\section{Irreducible component}

In order to remove components corresponding to 
periodic points of the half period, 
we define a quotient map. 
To formulate it rigorously, 
we need the following proposition.

\begin{proposition}\label{prop:div-map}
Let $f:\BR\to\BR$ be a $C^1$ map and $n$ a positive integer. 
Define a map $F_{2n}:\BR \to \BR$ as follows.
\[F_{2n}(x)=
\left\{ 
\begin{array}{ll}
\displaystyle\frac{f^{2n}(x)-x}{f^n(x)-x} & {\rm on\ } \ 
\{\,x\,|\, f^n(x)-x\ne0\,\} \\
\ & \ \\
(f^n)'(x)+1 & {\rm on\ } \ 
\{\,x\,|\, f^n(x)-x =0 \,\}
\end{array}\right.\]
Then $F_{2n}$ is continuous on $\BR$. 
\end{proposition}

\begin{proof}
Since $f$ is $C^1$, 
clearly $F_{2n}$ is continuous on $\{\,x\,|\, f^n(x)-x\ne0\,\}$ and 
on the interior of $\{\,x\,|\, f^n(x)-x=0\,\}$. 
We show that $F_{2n}$ is continuous at any point in 
$\overline{\{\,x\,|\, f^n(x)-x\ne0\,\}} \cap \{\,x\,|\, f^n(x)-x=0\,\}$. 

Let $p$ be an arbitrary point in 
$\overline{\{\,x\,|\, f^n(x)-x\ne0\,\}} \cap \{\,x\,|\, f^n(x)-x=0\,\}$. 
We write $g(x)=f^n(x)-x$. 
Then for $x\in\{\,x\,|\, f^n(x)-x\ne0\,\}$, 
\[F_{2n}(x) = \frac{g(x+g(x))-g(x)}{g(x)} +2 \ .\]
Since $g(x)\ne 0$ and $g$ is $C^1$ on $\BR$, by the mean value theorem, 
there is a $c_x\in (x,x+g(x))$ or $c_x\in (x+g(x),x)$ 
such that 
$$ \frac{g(x+g(x))-g(x)}{g(x)}=g'(c_x)\ .$$
When $x\to p$, we have $c_x \to p$ for $g(p)=0$. 
Since $g'(x)=(f^n)'(x)-1$ and it is continuous, 
\begin{eqnarray*}
\lim_{\scriptstyle x\to p \atop \scriptstyle f^n(x)\ne x} 
F_{2n}(x)
&=& 
\lim_{\scriptstyle x\to p \atop \scriptstyle f^n(x)\ne x} 
\left\{ \frac{g(x+g(x))-g(x)}{g(x)} +2 \right\} \\
&=& g'(p)+2 = (f^n)'(p)+1 = F_{2n}(p) \ .
\end{eqnarray*}
\end{proof}

In what follows, we fix a continuous full family 
of $C^1$ unimodal maps $\{ f_t \}_{t\in [0,1]}$. 
Similarly as in Proposition~\ref{prop:div-map}, 
we define its quotient map as follows.

\begin{definition}\label{rigorous_def_irreducible_map}
For any $m\in \BN$, 
we define a map $G_{m}:[0,1]\times \BR \to \BR$ as follows.
For odd $m=2n+1$, define 
$G_{2n+1}(t,x)=f_t^{2n+1}(x)-x$. 
For even $m=2n$, 
\[G_{2n}(t,x)=
\left\{ 
\begin{array}{ll}
\displaystyle\frac{f_t^{2n}(x)-x}{f_t^n(x)-x} & {\rm on\ } \ 
\{\,x\,|\, f_t^n(x)-x\ne0\,\} \\
\ & \ \\
(f_t^n)'(x)+1 & {\rm on\ } \ 
\{\,x\,|\, f_t^n(x)-x =0 \,\}
\end{array}\right.\]
We call $G_{m}$ the {\em quotient map} of 
$\{f_t\}_{t\in [0,1]}$ (although we do not divide $f_t^{m}(x)-x$ for odd $m$).
\end{definition}

Using this quotient map, we can define irreducible component of the bifurcation diagram as follows.

\begin{definition}\label{def_irreducible_map} 
\item[$(1)$] 
Let $m$ be a positive integer. 
We call each connected component $V$ of $G_m^{-1}(0)$ 
a {\em periodic point component} of $\{ f_t \}_{t\in [0,1]}$ , 
and the number 
$\max \{\, per(x)\,|\,  (t,x)\in V \,\} $ 
the {\em period} of $V$. 
\item[$(2)$] When $p$ is a periodic point of $f_t$ of minimal period $m$, 
we sometimes write that $(t,p)$ is a periodic point of minimal period $m$. 
\end{definition}

\begin{figure}[t]
\centerline{\includegraphics[width=13cm]{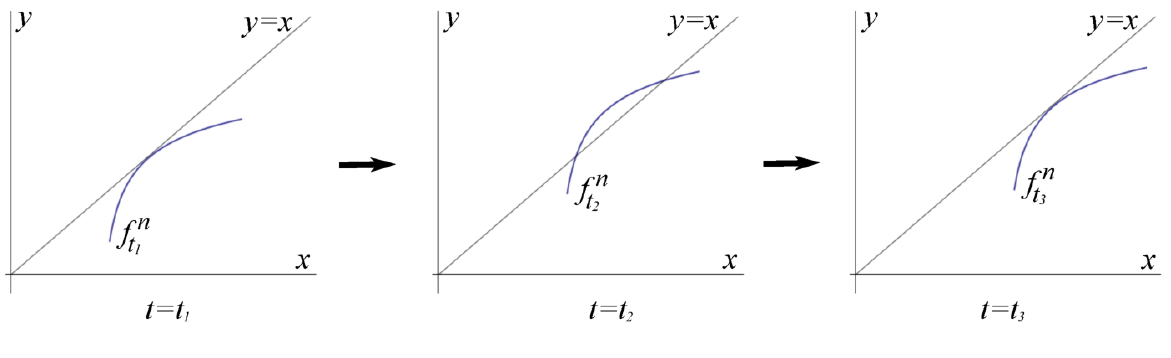}}
\caption{A process of making a closed component} 
\label{fig:graphs_t1t2t3}
\end{figure}

\begin{remark}\label{rem:periodic-point-component}
{\rm
\item[$(1)$] It is clear that $G_m^{-1}(0)$ is a subset of 
the bifurcation diagram of $\{ f_t \}$ of period $m$. 
\item[$(2)$] By the definition of $G_m$, 
$G_m^{-1}(0)$ contains all periodic points of the minimal period $m$, 
and possibly certain periodic points of lower period, 
but the period must be a divisor of $m$. 
\item[$(3)$]
If $m$ is odd, for any divisor $k$ of $m$ (including $1$), 
components of period $k$ are contained in $G_m^{-1}(0)$. 
When $m$ is even, 
$G_m^{-1}(0)$ can contain a periodic point $(t,x)$ of period $m/2$ 
(and its divisors) only when $( f_t^{m/2} )'(x)=-1$. 
\item[$(4)$]
Note that $G_m^{-1}(0) \subset [0,1] \times [-M,M] $, 
by Definition~\ref{def:full-family} (iii). 
Therefore, $G_m^{-1}(0)$ 
is a bounded set. 
By Proposition~\ref{prop:div-map}, 
$G_m$ is continuous as a function of $x$. 
We will show that $G_m$ is continuous on $[0,1]\times \BR$ 
in Proposition~\ref{prop:continuity-of-G2n}. 
Therefore, in $[0,1] \times \BR$, 
$G_m^{-1}(0)$ and each periodic point component are compact.  
\item[$(5)$]
\lq\lq Periodic point component" in Definition 5(1) is exactly what 
we called irreducible component. 
We did not define the term \lq\lq irreducible component" rigrously, but used it 
in the similar meaning in algebraic geometry. 
The reason why the periodic point component is suitable for being called irreducible 
is the following.

Let $V$ be a periodic point component of period $m$. 
We will see in Proposition~\ref{prop:k-is-the-half-of-n} that 
if the period of a point on $V$, say $k$, 
is smaller than $m$ then $k=m/2$. 
When $m$ is odd, the period of any point on $V$ is $m$, 
and $V$ does not intersect with components of lower periods. 
When $m$ is even, 
if the period of a component is not $m/2$, 
then it cannot intersect with $V$. 
And also the branchs of period $m/2$ are almost removed by the division 
(except for points $x$ such that $(f_t^n)'(x)=-1$). 
Thus, the periodic point component may be called irreducible. 

But there is still a question. 
Is there a possibility that period $m$ components of \lq\lq different types" 
intersect ?
That is exactly the main theme of this paper. 
We show that such intersections do not occur. 
\item[$(6)$]
There is a possibility of existence of many small periodic point components. 
If the shape of a part of the graph of $f_t^n$ changes as in 
Figure~\ref{fig:graphs_t1t2t3} when the parameter $t$ increases, 
then a small closed component appears as in Figure~\ref{fig:closed_component}. 
\begin{figure}[t]
\centerline{\includegraphics[width=5cm]{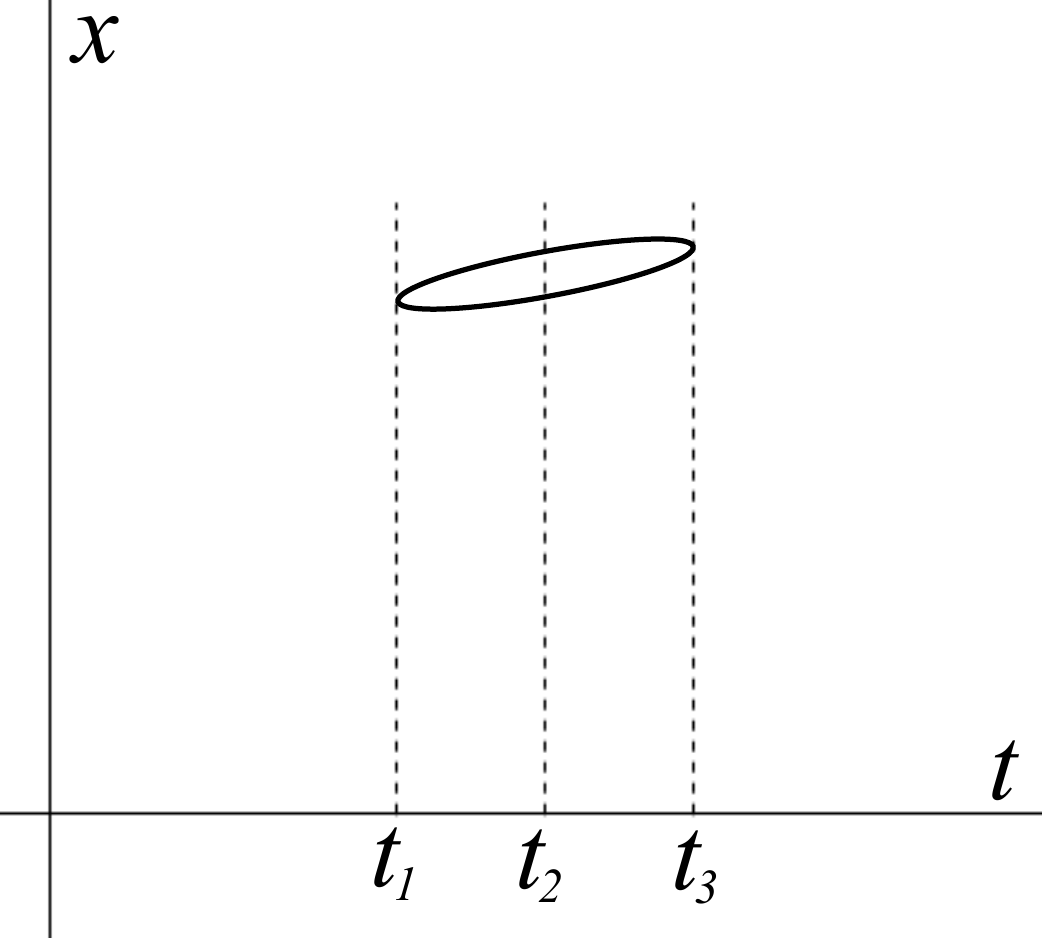}}
\caption{A closed component} 
\label{fig:closed_component}
\end{figure}

Suppose that the graph of $f_t^n$ has a wave intersecting to the line $y=x$ 
and converging to a point as in Figure~\ref{fig:infinite-waves-and-converging-components}(a). 
When $t$ increases, if the shape of this wave changes in a certain way,
then there is a possibility of existence of converging small component
to some component $V$ as in Figure~\ref{fig:infinite-waves-and-converging-components}(b). 
Note that the shapes of those small components may be a more deformed one,
and there may be more and more such small components in many places.
\begin{figure}[t]
\centerline{\includegraphics[width=11cm]{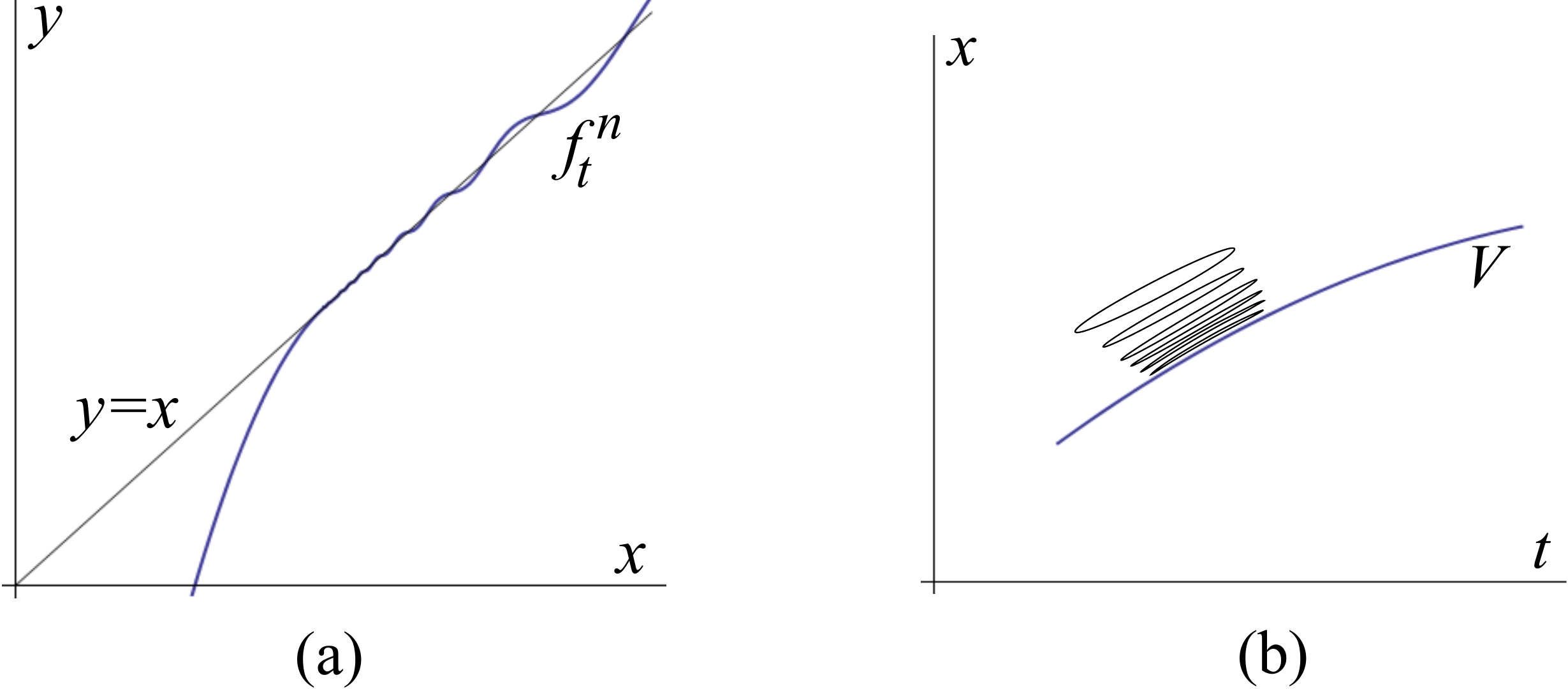}}
\caption{Converging infinite waves and converging infinite components} 
\label{fig:infinite-waves-and-converging-components}
\end{figure}

Our definition of unimodal map is so general that it allows
an interval of periodic points.
In such a case, we can not exclude non-pathwise connected 
periodic point component as shown in Figure~\ref{fig:non-pathwise-connected-component}. 
\begin{figure}[t]
\centerline{\includegraphics[width=5cm]{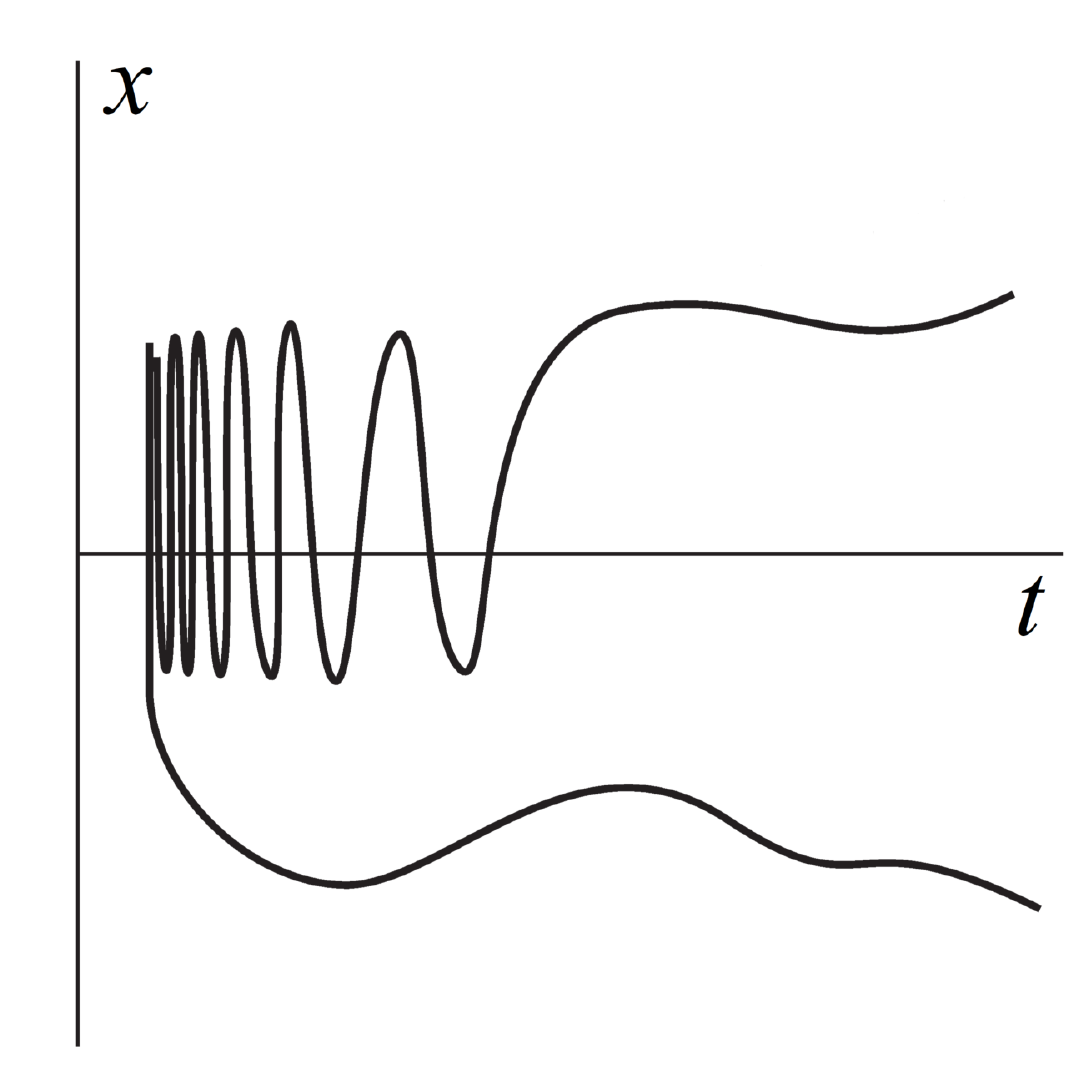}}
\caption{A non-pathwise connected component} 
\label{fig:non-pathwise-connected-component}
\end{figure}
Thus, the bifurcation diagram can be a very wild one. 
That is why we need a quite complicated argument of general topology to prove our main theorem.
}
\end{remark}

The continuity of $G_{m}$ is essential.

\begin{proposition}\label{prop:continuity-of-G2n}
$G_{m}$ is continuous. 
\end{proposition}

\begin{proof}
It is clear that $G_{2n+1}$ is continuous. 
We prove that $G_{2n}$ is continuous. 
We write $W=\{\,(t,x)\,|\, f_t^n(x)-x=0\,\} \subset [0,1]\times \BR$. 
It is clear that on $W^c=\{\,(t,x)\,|\, f_t^n(x)-x\ne 0\,\}$, 
$G_{2n}$ is continuous by the definition of $G_{2n}$. 
Also on the interior of $W$, $G_{2n}|_W$ is continuous, because 
$G_{2n}(t,x)=(f_t^n)'(x)+1$ on $W$ and $f_t$ is $C^1$. 
We shall prove that $G_{2n}$ is continuous at 
$(s,p)\in W \cap \overline{W^c}$. 
We have to show that when $(t,x)\to (s,p)$ for $(t,x)\in W^c$, 
$G_{2n}(t,x)\to G_{2n}(s,p)=(f_s^{n})'(p)+1$. 

Let $h_t(x)=f_t^n(x)-x$. Then
\begin{eqnarray*}
G_{2n}(t,x)
&=& 
\displaystyle\frac{f_t^{2n}(x)-x}{f_t^n(x)-x} \\
&=&
\displaystyle\frac{f_t^{n}\left( f_t^n(x)-x +x \right) -f_t^n(x)+ f_t^n(x)-x}{h_t(x)} \\
&=& 
\displaystyle\frac{f_t^{n}(x+h_t(x))- f_t^n(x)}{h_t(x)} \, + \, 1\\
\end{eqnarray*}
for $(t,x)\in W^c$. 
By the mean value theorem,
there exists a $c(t,x)\in (x,x+h_t(x))$ or $c(t,x)\in (x+h_t(x),x)$ such that 
\[\displaystyle\frac{f_t^{n}(x+h_t(x))- f_t^n(x)}{h_t(x)} =(f_t^n)'(c(t,x))\]
for $f_t^n$ is $C^1$. 
When $(t,x)\to (s,p)$, 
$(f_t^n)'(c(t,x))\to (f_s^n)'(p)$, because $h_s(p)=0$ and 
$g^1(t,x):=(f_t^{n})' (x)$ is continuous by Definition~\ref{def:continuous-family}. 
That means 
$G_{2n}(t,x)\to G_{2n}(s,p)=(f_s^{n})'(p)+1$. 
\end{proof}
\

The following are important properties of 
points of the half period of periodic point component.

\begin{proposition}\label{prop:diff-is-minus-one} 
Let $V$ be a periodic point component of period $2n$, 
and $V_n=\{ (t,x) \in V| f^n_t(x) - x = 0 \}$. Then, 
\begin{description}
\item[$(1)$] $(f^n_t)'(x)=-1$ for any $(t,x)\in V_n$. 
\item[$(2)$] For any $t\in [0,1]$, $V_n\cap \left( \{t\}\times \BR \right)$ is 
a finite set.
\end{description}
\end{proposition}

\begin{proof}
(1) follows from the definition of $G_{2n}$. 

By Proposition~\ref{prop:continuity-of-G2n}, 
$V$ is compact. Since $V_n$ is a closed subset of $V$, $V_n$ is also compact. 
We fix a $t\in [0,1]$. 
If $V_n\cap \left( \{t\}\times \BR \right)$  has an infinite number of points, 
then there is an accumulate point $(t,p)\in V_n$. 
However from (1), $(f^n_t)'(p)=-1$ and therefore $p$ is isolated in the 
set $\{x\in \BR| f^n_t(x) - x = 0\}$. That is a contradiction. 
\end{proof}

\section{Symbolic condition and the statement of Theorem~\ref{main-theorem}}\label{sec-statement}

In order to formulate our results rigorously, 
we need to give the symbolic condition for periodic points 
to be contained in the same periodic point component. 
First, let us recall some definitions and results on 
the kneading theory. 
There are two languages for kneading theory. 
One is Milnor-Thurston's invariant coordinate \cite{MT}, 
and another is Collet-Eckmann's $RL$-method \cite{CE} which is 
originated in \cite{MSS}. 
They are essentially equivalent. 
In this paper, we employ Collet-Eckmann's method because 
it is easier to imagine. 

Let $f$ be a unimodal map and $x\in\BR$. 
The {\em itinerary} $I(x,f)=A_0A_1A_2\cdots$ of $x$ for $f$ is 
the sequence of symbols $L$, $R$ and $C$ such that, 
\[ A_i= \left\{ \begin{array}{ll}
    L\quad & \mbox{if} \quad f^i(x)<0 \\
    R\quad & \mbox{if} \quad f^i(x)>0 \\
    C\quad & \mbox{if} \quad f^i(x)=0 
   \end{array} \right.  \]
When the symbol $C$ appears, the subsequent sequence is just 
the itinerary of $0$.
We shall omit the sequence after $C$. 
An infinite sequence of $R$ and $L$, 
or a finite sequence of $R$ and $L$ followed by $C$ 
is called an {\em admissible sequence}.
Thus, an itinerary is an admissible sequence.
In particular, the itinerary of the critical value $I(f(0),f)$ is 
called the {\em kneading sequence} of $f$ and denoted by $K(f)$.
As in \cite{CE}, 
we will write just $I(x)$ instead of $I(x,f)$, 
when the map acting on $x$ is clear in the context. 

A finite sequence of symbols $R$ and $L$ is 
called {\em even} ( resp. {\em odd} ) 
if it has an even ( resp. odd ) number of $R$'s. 
We define a natural ordering among admissible sequences. 
Let $A=A_0A_1A_2\cdots$ and $B=B_0B_1B_2\cdots$ be 
admissible sequences. 
We say $A<B$ if either $A_0\cdots A_{n-1}=B_0\cdots B_{n-1}$ 
is even and $A_n<B_n$, 
or it is odd and $A_n>B_n$, 
where we define $L<C<R$.
Let $\mathcal{S}(A_0A_1A_2\cdots)=A_1A_2\cdots$ be the shift map 
for admissible sequences. 
$\mathcal{S}(C)$ is not defined. 
It is clear that for any $x\in \BR$,
$\mathcal{S}(I(x))=I(f(x))$.

A sequence $A$ is called {\em maximal} if 
$\mathcal{S}^n(A)\le A$ for all $n\ge0$. 
The order of itineraries as admissible sequences is 
closely related to the order of corresponding points.

\begin{proposition}\label{prop:same-order}{\bf \cite{CE}} 
Let $f$ be a unimodal map. 
\begin{description}
\item[$(1)$] If $I(x)<I(x')$, then $x<x'$. 
\item[$(2)$] If $x<x'$, then $I(x)\le I(x')$. 
\end{description}
\end{proposition} 

By this proposition, 
we see that if $I(x)$ is maximal then $x\geq f^i(x)$ for any $i\in \BN$,
namely, $x$ is the biggest among the points of its orbit.

Let $f:[a,b] \to \BR$ be a horseshoe as in Definition 1, 
i.e. $a$ is a fixed point, $f(a)=f(b)=a$ and
there are $a_1$ and $b_1$ such that
$a<a_1<b_1<b$, $f(a_1)=f(b_1)=b$ and 
$|f'(x)|>1$ for any $x\in[a,a_1]\cup[b_1,b]$.
Then, it is a basic fact that $I:\Lambda(f)\to \Pi$ gives a topological conjugacy
between $f:\Lambda(f)\to \Lambda(f)$ and $\mathcal{S}:\Pi \to \Pi$,
where $\Lambda(f)=\bigcap_{n\geq 0}f^{-n}\left( [a,b] \right)$
and $\mathcal{S}:\Pi\to \Pi$ is the one-sided full shift of symbols $L$ and $R$.
In our case, $f_1$ is a horseshoe.
Therefore, there is a one-to-one correspondence between periodic points of $f_1$
and periodic sequences of $\mathcal{S}:\Pi\to \Pi$.

Now we define a map $\mu$ from the set of all periodic admissible 
sequences to itself. 
It plays a key role in our argument.

\begin{definition}  
Let $A=(A_1A_2\cdots A_n)^\infty$ be a periodic admissible sequence 
of minimal period $n$. 
By some shifts, 
$\mathcal{S}^m(A)=(A_{m+1}A_{m+2}\cdots A_{m+n})^\infty$ (~$0\le m\leq {n-1}$~) 
is a maximal sequence. 
We define 
\[\mu(A)=
\mathcal{S}^{n-m}(A_{m+1}A_{m+2}\cdots 
A_{m+n-1}{\overline A}_{m+n})^\infty\]

where ${\overline L}=R$ and ${\overline R}=L$, 
and we assume that $A_{n+i}=A_i$ for any $i$. 
\end{definition}

Note that in the definition of $\mu$, 
minimal period of $A$ is important. 
Since there is a unique $m$ ($0\le m \le n-1$) 
such that $(A_{m+1}A_{m+2}\cdots A_{m+n})^\infty$ is maximal, 
this definition is well-defined. 
It is clear that;

\begin{proposition}\label{prop:shift-of-mu} 
Let $A$ be a periodic admissible sequence. 
For any $k\in\BN$, 
$\mu(\mathcal{S}^k(A)) = \mathcal{S}^k(\mu(A))$. 
\end{proposition}

If $A=(A_1 \cdots A_n)^\infty$ is maximal, 
then $\mu(A)=(A_1 \cdots A_{n-1} \overline{A_n})^\infty$.
As mentioned above,
if  $I(x)=A=(A_1 \cdots A_n)^\infty$ for a periodic point $x$ of $f$ is maximal, 
then $x$ is the biggest among its orbit.
The symbol $A_n$ corresponds to the previous point of the orbit of $x$ i.e. $f^{n-1}(x)$. 
It is located near the critical point $0$ and it moves either from left to right or from right to left
passing through $0$ after either saddle-node bifurcation or period doubling bifurcation.
That is the meaning of the map $\mu$.
If we identify those periodic sequences with periodic points of a unimodal map 
having the sequences as their itineraries, 
then $\mu$ maps the periodic point to the twin  
which was born at the same time through a saddle-node bifurcation, 
in the case where $per(\mu(A))=per(A)$. 
In some cases, $per(\mu(A))$ can be smaller than $per(A)$. 
It is clear that it must be a divisor of $n$, 
but actually it must be exactly $n/2$ 
when it is really smaller than $n$. 
The following is the proof, 
but we give a more general statement for later use.

\begin{proposition}\label{prop:take-mu}
Let $A=(A_1\cdots A_n)^{\infty}$ be a periodic admissible sequence of period $n$ 
(not necessarily of minimal period $n$). 
\begin{description}
\item[$(1)$] 
If $A=(A_1\cdots A_n)^{\infty}$ is maximal, 
then the minimal period of
\newline
$\widetilde{A}=(A_1\cdots A_{n-1}\overline{A_n})^{\infty}$ 
is either $n$ or $n/2$. 
\item[$(2)$] If $A=(A_1\cdots A_n)^{\infty}$ is maximal, then 
either $per(A)$ or $per(\widetilde{A})$ is $n$. 
\item[$(3)$] If $per(A)=n$, then $per(\mu(A))$ is either $n$ or $n/2$. 
\end{description}
\end{proposition}

\begin{proof} 
\item[(1)] 
Suppose that the minimal period of $\widetilde{A}=(A_1\cdots A_{n-1}\overline{A_n})^{\infty}$ 
is smaller than $n/2$. 
Then it is written as 
$A_1\cdots A_{n-1}{\overline A}_n=B\cdots B$, 
where $B$ is a finite sequence of $R$ and $L$. 
We denote the length of $B$ and the number of $B$'s 
by $m$ and $k$ respectively, then $k\ge3$ and $km=n$.
Let $B=B_1\cdots B_m$ and ${\widetilde B}=B_1\cdots B_{m-1}{\overline B}_m$. 
Then $A_1\cdots A_n=B\cdots B{\widetilde B}$, where the number of 
$B$'s is $k-1$. 
Since $A$ is maximal and $k-1\ge2$, by shifting $A$, we have 
$BB>B{\widetilde B}$ and $B>{\widetilde B}$.  
The second inequality means that $B$ is odd. 
However that contradicts the first inequality. 
\item[(2)] Suppose that $per(\widetilde A)\ne n$. 
Then by (1), $per(\widetilde A)=n/2$, and therefore we can write 
$A_1\cdots A_{n-1}{\overline A}_n=BB$ for some finite sequence $B$ of 
$R$ and $L$. 
Note that the $n$ is even in this case. 

Assume that $per(A)=k \ne n$. 
Firstly, we claim that $n/k$ is odd. 
Because if $n/k$ is even, we can write $A=D \cdots D$ where 
$D$ is a sequence of $R$ and $L$ of length $k$, and the number of $D$'s is even. 
However that contradicts $\widetilde{A}=BB$. 
Now suppose that $n/k$ is odd. 
Note that $n/k\geq 3$ because $k \ne n$. 
Since the number of $D$'s is odd and $n$ is even, 
the length of $D$ is even. 
Therefore, we can write $D=D'D''$ where the lengths of $D'$ and $D''$ 
are the same. 
Since $A_1\cdots A_{n-1}{\overline A}_n=BB$, 
the first $B$ must start with $D'$, 
and the second $B$ must start with $D''$. 
Therefore $D'=D''$ and $D=D'D'$. 
Then $B$ is a sequence of odd number of $D'$'s, 
and $A=D\cdots D$ is a double of $B$. 
That means $A=BB= \widetilde{A}$ and leads to a contradiction.
\item[(3)] Let $per(A)=n$. By some shift, 
we can assume that $A$ is maximal. 
Then by (1), $per(\mu(A))=n/2$ if $per(\mu(A))<n$. 
\end{proof}

By the previous proposition,
we have two situations when we apply $\mu$ on a periodic sequence.
One is that we get a sequence whose minimal period 
is the same as the original one. 
In section~\ref{sec-map-mu}, we will see that $\mu(\mu(A))=A$ in this case. 
Another case is that the minimal period of $\mu(A)$ is exactly the half. 
In this case, $\mu(A)$ indicates the itinerary of the periodic point 
from which the periodic point corresponding to $A$  
was born through a period-doubling bifurcation. 
What we want to get by applying $\mu$ is the itinerary of the 
periodic point which is on the same periodic point component of 
the bifurcation diagram. 
In case of period-doubling bifurcation, 
that is a point which is in the same orbit and obtained from the 
first one by applying the map the half times of the period. 
Thus, in this case, we need another treatment of the sequence.

\begin{definition}\label{Def:nu}
We define a map $\nu$ from the set of all periodic admissible sequences 
to itself as follows.
Let $A=(A_1A_2\cdots A_n)^\infty$ be a periodic admissible sequence 
of minimal period $n$.
When $per(\mu(A))=n$, define $\nu(A)=\mu(A)$.
When $per(\mu(A))$ is smaller than $n$ ( that is $n/2$ ), 
define $\nu(A)=\mathcal{S}^{n/2}(A)$. 
\end{definition}

For the standard family of quadratic maps, 
by the monotonicity of kneading sequences \cite{MT} and the 
non-degeneracy of bifurcation \cite{DH}, 
we can see that expanding periodic points $p$ and $q$ of 
minimal period $n$ are on the same periodic point component 
if and only if $I(p)=\nu(I(q))$. 

Our main theorem in this paper asserts that the same is true 
for any continuous full family of $C^1$ unimodal maps.

\begin{theorem}\label{main-theorem}  \quad 
Let $(1,p), (1,q)\in G_n^{-1}(0)\cap\{t=1\}$ $( p\ne q )$ and 
$per(p)=per(q)=n$. 
Then, 
$(1,p), (1,q)$ are on the 
same periodic point component if and only if $I(p,f_1)=\nu(I(q,f_1))$. 
\end{theorem}

\begin{example}
{\rm
Let us consiter the standard family of quadratic maps $q_t(x)=t-x^2$ and 
its bifurcation diagram of period $6$. 
Refer Figure~\ref{figure_2}. 
Figure~\ref{fig:BD-Period-6-magnified} is the 
$[1.4, 2.1]\times [-2,2]$ part of Figure~\ref{figure_2}. 
The component of fixed points and the component of period $2$ are removed. 

\begin{figure}[t]
\centerline{\includegraphics[width=8cm]{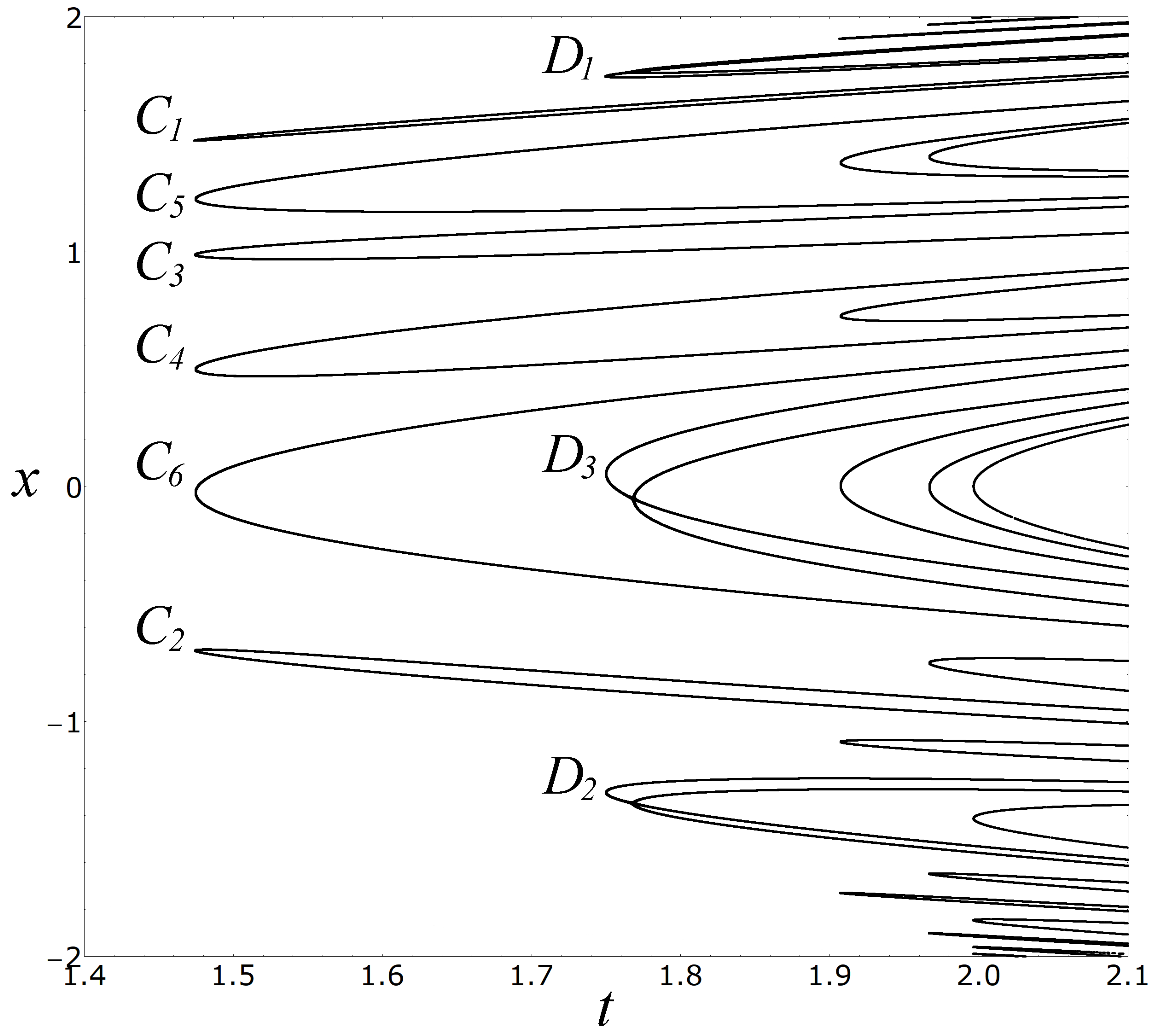}}
\caption{A magnified bifurcation diagram}
\label{fig:BD-Period-6-magnified}
\end{figure}

When $t$ is approximately $1.47$, 
a saddle-node bifurcation occurs and the first period orbit of period 6 apears. 
That produces the components $C_1, \ldots, C_6$. 
Let $q(t,x)=(t,q_t(x))$. 
Then $q(C_i)=C_{i+1}$ for $1\leq i \leq 5$ and $q(C_6)=C_1$. 
The itineraries of two points at the right end of $C_1$ are 
$(RLRRRR)^\infty$ and $(RLRRRL)^\infty$. 
($C_1$ is too thin and we can not see the shape very well. 
But the shape is basically the same as that of other $C_i$'s.)
It is natural that those two itineraries are maximal, 
because any point on $C_1$ has the biggest $x$-value among its orbit. 

It is known that 
for any $t>2$, any periodic point of $q_t$ is expanding 
and its itinerary does not change when $t$ increases. 
Therefore, we can discuss the itineraries of the periodic points 
of $q_{2.1}$ on the right end of Figure~\ref{fig:BD-Period-6-magnified}  
although it is not a horseshoe yet. 

Let $x_1>x_2$ be two right end points on $C_1$. 
Then $I(x_1)=(RLRRRR)^\infty$ and 
$\mu(I(x_1))=(RLRRR \overline{R})^\infty =(RLRRRL)^\infty=I(x_2)$. 
Since the period of $\mu(I(x_2))$ is also $6$, by the Definition~\ref{Def:nu}, 
we have $\nu(I(x_1))=\mu(I(x_1))=I(x_2)$. 
Also for any other component $C_i$, the situation is the same. 
Since $q^i(C_1)=C_{i+1}$ for $1\leq i \leq 5$, 
the itineraries of the right end two points of $C_i$ are $I(x)$ and $\mu(I(x))$ 
each other. 

The components $D_1$, $D_2$, $D_3$ apear when $t$ is approximately $1.75$. 
They are components of the periodic orbit of period $3$. 
Note that $q(D_1)=D_2$, $q(D_2)=D_3$ and $q(D_3)=D_1$. 
The itineraries of two right end points of $D_1$ are $(RLL)^\infty$ and $(RLR)^\infty$, 
and they are maximal. 
(Also $D_1$ is too thin, but the shape is basically the same as $D_3$ and $D_2$.)
Note that the period $3$ part of $D_i$ ($i=1,2,3$) is removed by the division. 
Therefore, it is not contained in $G_6^{-1}(0)$. 

A period-doubling bifurcation occurs when $t$ is approximately $1.77$. 
On $D_1$, it occurs on $(RLR)^\infty$ side, 
and as a result, a period $6$ branch is born. 
Its itinerary is $(RLRRLL)^\infty$ and the period is $6$. 
The periodic point correspoing to 
the $3$-times shift $\mathcal{S}^3((RLRRLL)^\infty)=(RLLRLR)^\infty$ 
is on the same period $6$ component. 
$(RLLRLR)^\infty$ is maximal. 
If we apply $\mu$ to $(RLLRLR)^\infty$, 
the period of $\mu((RLLRLR)^\infty)=(RLLRLL)^\infty=(RLL)^\infty$ becomes $3$. 
In this case, by the Definition~\ref{Def:nu}, 
$\nu((RLLRLR)^\infty)=\mathcal{S}^3((RLLRLR)^\infty)=(RLRRLL)^\infty$ 
and the corresponding point is on the same period $6$ branch. 
On $D_2$, $D_3$, the situation is similar 
because they are the mapped images of $D_1$. 
}
\end{example}

In the following sections, we prove some preliminary results. 
The proof of Theorem~\ref{main-theorem} will be given in 
section~\ref{sec:proof-main-theorem}.

\section{Basic properties of the map $\mu$}\label{sec-map-mu}

The proof of the following proposition is straightforward. 
Also this proposition is essentially the same as Lemma 3 of \cite{JR2}. 
Although some translation work is necessary, 
we can see that the definition of \lq\lq admissible\rq\rq \  in \cite{JR2} 
is the same as our \lq\lq maximal\rq\rq.

\begin{proposition}\label{prop:maximal-sequence}
Let $A=(A_1 \cdots A_n)^\infty$ be a maximal sequence 
of minimal period $n$. 
Then $\widetilde{A}=(A_1 \cdots A_{n-1} {\overline A}_n)^\infty$ is also 
a maximal sequence. 
\end{proposition}

Note that if $per(A)<n$, then this proposition does not hold. 
For example, $A=(RLLRLLRLL)^\infty$ is maximal. 
However $\widetilde{A}=(RLLRLLRLR)^\infty$ is not maximal. 

From this proposition, 
it follows easily that;

\begin{proposition}\label{prop:mu-is-involutive} 
Let $A$ be a periodic sequence of minimal period $n$. 
If $per(\mu(A))=n$, then $\mu(\mu(A))=A$. 
\end{proposition}

The following is one of the important properties of the map $\mu$.

\begin{proposition}\label{prop:monotonicity-of-mu}
Let $A$ and $B$ be periodic sequences of minimal period $n$. 
If $\mu(A)=\mu(B)$, then either $A=B$ or, 
$n$ is even and $\mathcal{S}^{n/2}(A)=B$. 
\end{proposition}

\begin{proof}
If $per(\mu(A))=per(\mu(B))=n$, 
then by Proposition~\ref{prop:mu-is-involutive}, 
$A=\mu(\mu(A))=\mu(\mu(B))=B$. 
So, we assume that 
$per(\mu(A))=per(\mu(B))=n/2$. 

Let $A=(A_1 \cdots A_n)^\infty$ and $B=(B_1 \cdots B_n)^\infty$. 
Let $k$ and $h$ be integers such that 
$1\le k\le n$, $1\le h \le n$ and 
$$\begin{array}{ll} 
\mu(A) & =(A_1 \cdots {\overline A}_k \cdots A_n)^\infty \\
\mu(B) & =(B_1 \cdots {\overline B}_h \cdots B_n)^\infty \ .
\end{array}$$
If $k=h$, then $A=B$ for $\mu(A)=\mu(B)$. 
So, without loss of generality, we assume that $k>h$. 
By the definition of $\mu$ and Proposition~\ref{prop:maximal-sequence}, 
$\mathcal{S}^k(\mu(A)) =   (A_{k+1} \cdots A_n A_1 \cdots {\overline A}_k)^\infty$ and 
$\mathcal{S}^h(\mu(B)) = (B_{h+1} \cdots B_n B_1 \cdots {\overline B}_h)^\infty$ are 
maximal. 
Since the maximal sequence of periodic sequence is unique 
and $\mu(A)=\mu(B)$, we have 
$\mathcal{S}^k(\mu(A)) = \mathcal{S}^h(\mu(B))$. 
Since $\mu(A)=\mu(B)$ again, we have, 
$$\mathcal{S}^k(\mu(A))=\mathcal{S}^h(\mu(B)) = \mathcal{S}^h(\mu(A)) \ . $$
Therefore, 
$\mathcal{S}^{k-h}(\mu(A))=\mu(A)$. 
Then $k-h= n/2$, because 
$1\le k-h <n$ and $per(\mu(A))= n/2$. 
Also $(A_{k+1} \cdots A_n A_1 \cdots {\overline A}_k)^\infty
=(B_{h+1} \cdots B_n B_1 \cdots {\overline B}_h)^\infty$ means that 
$\mathcal{S}^k(A) = \mathcal{S}^h(B)$. 
So we get 
$$\mathcal{S}^h(B)=\mathcal{S}^k(A) = \mathcal{S}^{(n/2) + h}(A) $$
and this means 
$\mathcal{S}^{n/2}(A)=B$. 
\end{proof}

\section{Periodic point components}\label{sec:periodic-point-components}

In the definition of periodic point components, 
we removed components corresponding to the period $n/2$ 
from ${\widetilde G}_n^{-1}(0)$ when $n$ is an even number. 
Note that for a divisor $k$ of $n$ which is not a divisor of $n/2$, 
${\widetilde G}_k^{-1}(0)$ is not removed from 
${\widetilde G}_n^{-1}(0)$. 
However, we can show that such components 
do not have any intersection with components of minimal period $n$. 
In fact, we prove the following proposition which asserts that 
for any continuous family of $C^1$ maps on $\BR$ 
( not only of unimodal maps ), 
on the bifurcation diagrams, minimal period changes only to 
the half when it decreases.

\begin{proposition}\label{prop:period-changes}
Let $g_n \ (n=1,2,\ldots )$ be a sequence of $C^1$ maps 
converging to a $C^1$ map $g$ in the $C^1$ topology. 
Let $\mathcal{O}_n=\{q_1(n) < q_2(n) < \cdots\ < q_m(n)\}$ 
be a periodic orbit of $g_n$ of minimal period $m$ 
such that $q_i(n)$ converges to a periodic point $q_i$ of $g$ 
as $n\to \infty$ for each $i$. 
We write $\mathcal{O}=\{q_1\le q_2\le \cdots \le q_m\}$. 
If the minimal period of $\mathcal{O}$, say $k$, 
is smaller than $m$, then $k=m/2$.
\end{proposition}

\begin{proof}
First of all, 
we claim that $\mathcal{O}$ does not contain any critical point. 
If one of those points is a critical point, 
then $\mathcal{O}$ is a hyperbolic sink of minimal period $k<m$. 
Since it is stable under small $C^1$ perturbation, 
for sufficiently large $n$, 
$g_n$ turns out to have a sink of the same period near $\mathcal{O}$, 
and its basin is very close to that of $\mathcal{O}$. 
That means $g_n$ cannot have a periodic point 
of period $m$ near $\mathcal{O}$. 
However that is a contradiction. 

We write $\mathcal{O} =\{ q_1 \le \cdots \le q_m \} 
= \{ p_1 < \cdots < p_k \}$. 
If $n$ is sufficiently large, there are numbers 
$1=i_1<i_2< \cdots <i_k <i_{k+1} = m+1 $ such that 
if $i_j\le h<i_{j+1}$ then $q_h(n)$ is very close to $p_j$ 
for all $1\le j\le k$. 

We write $I_j(n)=[q_{i_j}(n),q_{i_{j+1}-1}(n)]$ for all $1\le j\le k$. 
Since there is no critical point in $\mathcal{O}$, 
for sufficiently large $n$, 
$g_n$ is a local homeomorphism near $\mathcal{O}$. 
In particular, $g_n$ must map each interval $I_j(n)$ onto 
one of others homeomorphically. 
That means the set of all the boundary points of $I_j(n)$'s is 
invariant under $g_n$. 
Since $\mathcal{O}_n$ is a periodic orbit of minimal period $m$, 
$\mathcal{O}_n$ must consist of only the boundary points. 
Therefore, we have $k=m/2$. 
\end{proof}

\begin{proposition}\label{prop:k-is-the-half-of-n}
Let $V$ be a periodic point component.
If the period of $V$ is $n$ and 
$k=\min\{ per(t,x)| (t,x)\in V\}$ is smaller than $n$, 
then $k=n/2$ . 
\end{proposition}

\begin{proof}
Suppose that $k\ne n/2$. 
Since $k$ must be a divisor of $n$, 
we have $k<n/2$. 
We write 
$V'=\{ \, (s,x)\in V\, |\, per(x)=n \mbox{ or } n/2\,\}$ 
and $V''=V-V'$. 
$V''$ is a closed subset of $V$, 
because $per(x)<n/2$ for any $(s,x)\in V''$. 

It is clear that there exists a point 
$(s_0,q_0)\in V''$ such that 
$(s_0,q_0)\in \overline{V'}$, otherwise 
$V''$ is a closed and open subset of $V$, 
and $V$ turns out to be non-connected. 
Therefore, there is a sequence of points 
$(s_i,p_i)\in V'$ ($i=1,2,\ldots $) such that 
$(s_i,p_i)\to (s_0,q_0)$ as $i\to \infty$.  
If $per(p_i)=n$ for infinitely many $i$'s, 
then by Proposition~\ref{prop:period-changes}, 
$per(q_0)$ must be $n/2$, 
and that contradicts the definition of $V''$. 
So, we can assume that $per(p_i)=n/2$ 
for any $i\ge1$. 
Then by Proposition~\ref{prop:period-changes} again, 
we get $per(q_0)=n/4$. 
Since $f_{s_0}^{n/4}(q_0)=q_0$, we have 
$(f_{s_0}^{n/2})'(q_0)=((f_{s_0}^{n/4})'(q_0))^2 >0$. 
However that contradicts Proposition~\ref{prop:diff-is-minus-one}(1). 
\end{proof}
\ 

One of the keys to the proof of out main theorem 
is Proposition~\ref{prop:two-points} 
which claims that every periodic point component has two and more intersection 
points with $\{t=1\}$ line. 
(In fact, we will see later that the number of intersection points is 
exactly two.)
The essence of the proof of Proposition~\ref{prop:two-points} is 
the following Proposition~\ref{prop:connecting-curve} whose proof needs 
a quite involved argument of general topology.

\begin{proposition}\label{prop:connecting-curve}
Let $M>0$ be a positive number, $D=[0,1]\times[-M,M]$ a rectangle and 
$G:D\to \BR$ a continuous map such that 
$G^{-1}(0) \cap \partial D$ is a finite set on $\{1\}\times (-M,M)$. 

If there is a connected component $F_C$ of $G^{-1}(0)$ which intersects 
exactly once to $\{ 1 \} \times (-M, M)$ at a point $p_0 = (1, y_0)$, 
then there is a path in $D - G^{-1}(0)$ connecting 
points $y_-$ and $y_+$ in $\{ 1 \} \times (-M, M)$ which are 
arbitrary close to $y_0$ with $y_- < y_0 < y_+$. 
\end{proposition}

This proposition might be seemingly obvious. 
But actually, we need a quite complicated argument for the proof, 
because there is a possibility of the existence of infinitely many 
components accumulating to other components as mentioned in 
Remark~\ref{rem:periodic-point-component}(6). 
Such accumulation can occur in many places and also the shapes of such components 
can be a quite deformed one. 
For those reasons, the existence of a continuous curve as stated in 
Proposition~\ref{prop:connecting-curve} is not obvious.

We will give the proof in section~\ref{section-proof-connecting-curve}. 
In our situation, $G_n^{-1}(0)\cap \partial D$ is a finite set 
on $\{1\}\times (-M,M)$ for any $n$, because $f_1$ is a horseshoe and $f_t$ does not 
have any periodic point on $\partial D$ except them by hypothesis.

\begin{proposition}\label{prop:two-points}
Let $V$ be a periodic point component of $G_n^{-1}(0)$ of period $n$. 
If $V\cap\{t=1\}\ne\emptyset$, then $V\cap\{t=1\}$ contains at least two points. 
\end{proposition}

\begin{proof}
Suppose that $V\cap\{t=1\}=\{ (1,p) \}$ is only one point. 
Note that the minimal period of $p$ is $n$, because 
by Proposition~\ref{prop:k-is-the-half-of-n}, 
$per(p)$ is $n$ or $n/2$, and if $per(p)=n/2$, 
then by Proposition~\ref{prop:diff-is-minus-one}(1), 
$\left(f_1^{n/2}\right)'(p)=-1$. 
However $f_1$ is horseshoe and periodic points must be hyperbolic. 
Therefore $per(p)=n$. 

Let $J$ be an open interval in $\{t=1\}$ such that $p\in J$ and 
$J\cap G_n^{-1}(0)=\{p\}$. 
Then by Proposition~\ref{prop:connecting-curve}, 
there are two points $p_0<p<p_1$ in $J$ and a continuous curve 
$\gamma:[0,1]\to [0,1]\times [-M,M]$ such that $\gamma(0)=p_0$, $\gamma(1)=p_1$ 
and $\gamma([0,1])\cap G^{-1}(0)=\emptyset$. 

When $n$ is odd, $G_n(t,x)=f_t^n(x)-x$ and 
$$\frac{dG_n}{dx}(1,p)=(f^n_1)'(p)-1\ .$$ 
When $n$ is even, 
note that $f_1^{n/2}(p)-p\ne0$ because the minimal period of $p$ is $n$. 
By an easy calculation, we have 
$$ \frac{dG_n}{dx}(1,p) = 
\frac{(f^n_1)'(p)-1}{f^{n/2}_1(p)-p}  \  .$$
Since $p$ is a hyperbolic periodic point of $f_1$, 
$(f^n_1)'(p)\ne 1$. 
Therefore, in both cases, $\displaystyle\frac{dG_n}{dx}(1,p)\ne0$, 
and the signs of $G_n(1,p_0)$ and $G_n(1,p_1)$ are different. 
However, since $\gamma$ does not have any intersection 
with $G_n^{-1}(0)$, 
$G_n(\gamma(s))$ is non-zero for any $0\le s\le1$. 
That is a contradiction. 
\end{proof}

\section{Shuffle of periodic point}

For a positive integer $n$, 
we denote the set of all permutations of \mbox{$\{0,1, \ldots ,n-1\}$} 
by $\mathcal{S}_n$.

\begin{definition}
Let $(i_1,i_2, \ldots ,i_n)\in\mathcal{S}_n$ and 
$f:\BR \to \BR$ be a map. 
For a periodic point $p$ of $f$ with minimal period $n$, when 
\[f^{i_1}(p)<f^{i_2}(p)< \cdots <f^{i_n}(p)\ ,\]
we denote  $\sigma(p)=(i_1,i_2, \ldots ,i_n)$ 
and call it the {\em shuffle} of $p$. 
\end{definition}

\begin{proposition}\label{prop:shuffle}
Let $f$ and $g$ be unimodal maps. 
Let $p$ and $q$ be periodic points of $f$ and $g$ respectively 
such that 
\begin{description}
\item[$(i)$] $per(p)=per(q)=n$, 
\item[$(ii)$] the orbits of $p$ and $q$ do not contain the critical point, 
\item[$(iii)$]  $\sigma(p)=\sigma(q)$. 
\end{description}
Then at least one of the following holds: 
\begin{description}
\item[$(a)$] $I(p)=I(q)$, 
\item[$(b)$] $per(I(q))=n$ and $I(p)=\mu(I(q))$, 
\item[$(c)$] $per(I(p))=n$ and $I(q)=\mu(I(p))$. 
\end{description}
\end{proposition}

\begin{proof}
Assume that $\sigma(p)=\sigma(q)=(i_1, \ldots ,i_n)$. 
Let $x$ stands for $p$ or $q$, and 
$h$ stands for $f$ or $g$ corresponding to $x$.

If $n=1$, then $I(x)=R^\infty$ or $L^\infty$ because 
neither $p$ nor $q$ is the critical point. 
So, the statement clearly holds. 
We assume $n\ge2$. 

By the definition, 
$h^{i_n}(x)$ is the biggest among $h^{i_j}(x)$'s. 
We write $i_k=i_n-1$.
When $i_n=0$, we define $i_k=n-1$. 
Then $h(h^{i_k}(x))=h^{i_n}(x)$. 

Since $i_k \ne i_n$, $h^{i_{k+1}}(x)$ exists. 
Therefore, by the definition of shuffle, 
\[h^{i_{k-1}}(x)<h^{i_k}(x)<h^{i_{k+1}}(x)\]
if $h^{i_{k-1}}(x)$ exists. 

We have $h^{i_{k+1}}(x)>0$, 
because if $h^{i_{k+1}}(x)<0$ 
then $h^{i_{k}}(x)<h^{i_{k+1}}(x)<0$, 
and since $h$ is orientation preserving on $\{x<0\}$, 
that contradicts the fact that 
$h^{i_n}(x)$ is the biggest among $h^{i_j}(x)$'s. 
If $k\ne 1$, $h^{i_{k-1}}(x)$ exists and 
$h^{i_{k-1}}(x)< h^{i_{k}}(x)$. 
We have $h^{i_{k-1}}(x)<0$, 
 because if $h^{i_{k-1}}(x)>0$ 
then $0<h^{i_{k-1}}(x)<h^{i_{k}}(x)$, 
and since $h$ is orientation reversing on $\{x>0\}$, 
that contradicts the fact that 
$h^{i_n}(x)$ is the biggest among $h^{i_j}(x)$'s. 

Those facts mean that only $h^{i_k}(x)$ has the freedom of 
$R$ or $L$. 
Note that the orbits of $p$ and $q$ do not have the critical point, 
and symbol $C$ does not appear.

If $I(p)\ne I(q)$, 
then only the symbols corresponding to the points $f^{i_k}(p)$ and $g^{i_k}(q)$ are different.
By Proposition~\ref{prop:same-order}, 
$I(f^{i_n}(p))$ and $I(g^{i_n}(q))$ are maximal sequences. 
If $I(f^{i_n}(p)) \ne I(g^{i_n}(q))$, 
then by Proposition~\ref{prop:take-mu} (2), 
either $per\left( I(f^{i_n}(p)) \right)$ or $per\left( I(g^{i_n}(q)) \right)$ 
is $n$. 
Therefore we have that 
$per(I(q))=n$ and $I(p)=\mu(I(q))$, 
or $per(I(p))=n$ and $I(q)=\mu(I(p))$ 
by the definition of $\mu$.  
\end{proof}

Note that $per(I(p))$ or $per(I(q))$ can be $n/2$. 
The property \lq\lq same shuffle\rq\rq\  is an open property 
on components of constant period, namely;

\begin{proposition}\label{prop:same-shuffle}
\item[$(1)$] Let $(t,p)$ be a periodic point of minimal period $n$. 
Then there exists a $\delta>0$ such that if $(s,q)$ is a periodic point 
of minimal period $n$ and $(s,q)\in B(\delta,(t,p))$ then 
the shuffles of $(t,p)$ and $(s,q)$ are the same, 
where $B(\delta,(t,p))$ is the $\delta$-disk in $\BR^2$ centered at $(t,p)$. 
\item[$(2)$]
Let $W$ be a connected subset of a periodic point component. 
If $per(x)$ is the same for any $(t,x)\in W$, 
then the shuffle $\sigma(x)$ is the same for any $(t,x)\in W$. 
\end{proposition}

\begin{proof}
\item[(1)] 
Let $p_1<\cdots<p_n$ be the orbit of $p$ and 
$B(\epsilon,p_i)$ the $\epsilon$-neighborhood of $p_i$ in $\BR$. 
It is clear that there is an $\epsilon>0$ satisfying the following properties.
\begin{description}
\item[(i)] $B(\epsilon,p_i)$'s are disjoint. 
\item[(ii)] For any $B(\epsilon,p_i)$, 
$f_t(B(\epsilon,p_i))$ has an intersection with only $B(\epsilon, f_t(p_i))$. 
\end{description}
Then, also it is clear that if we take a $\delta>0$ small enough, 
any periodic point $(s,q)$ of period $n$ contained in $B(\delta, (t,p))$ satisfies 
the following. 
\begin{description}
\item[(i)] Let $q_1<\cdots <q_n$ be the orbit of $q$. 
Then $q_i \in B(\epsilon/2,p_i)$ for any $i$. 
\item[(ii)] $f_s(q_i) \in B(\epsilon/2,f_t(p_i))$ for any $i$. 
\end{description}
That means that the shuffles of $(s,q)$ and $(t,p)$ are the same. 

\item[(2)] 
Since there are only a finite of number of shuffles 
for a fixed period, 
$W$ can be divided into a finite number of disjoint subsets 
$W_1, \ldots , W_k$ such that the shuffle is the same on each $W_i$. 
(1) means that each $W_i$ is an open set because the period is the same on $W$. 
Therefore, $W$ is a disjoint union of open sets $W_i$'s. 
For $W$ is connected, we have $k=1$, namely, 
the shuffle is constant on $W$. 
\end{proof}

The main point in the proof of our theorem is that 
if a periodic point component $V$ of period $n$ contains only points of 
period $n$ and $n/2$, 
then shuffles of points of period $n$ in $V$ must be either 
the same or $n/2$-shift each other. 
Let us define $n/2$-shift of shuffle more precisely.

\begin{definition}
Let $\sigma=(i_1, \ldots ,i_n) \in \mathcal{S}_n$ and $n$ be even. 
We define $n/2$-shift of $\sigma$ by 
$\gamma(\sigma)=([i_1+n/2], \ldots ,[i_n+n/2])$, 
where $[k]=k \ {\rm mod} \ n$. 
\end{definition}

Note that if $x$ is a periodic point of period $n$ and 
$n$ is even, 
then $\gamma(\sigma(x)) = \sigma(f^{n/2}(x))$. 
Also, it is clear that $\gamma(\gamma(\sigma))=\sigma$. 
Therefore, we can define that $\sigma$ is equivalent to $\rho$  
if either $\sigma=\rho$ or $\gamma(\sigma)=\rho$. 
We denote the equivalence class of $\sigma$ by $[\sigma]$, 
and call it the {\em shuffle class} of $\sigma$. 
The following proposition is essential in the proof of our theorem.

\begin{proposition}\label{prop:half-shift}  
Let $n$ be even, $V$ a periodic point component of period $n$ and 
$\min\{\,per(x) \,|\, (t,x)\in V\,\}=n/2$. 
\begin{description}
\item[$(1)$] For any $(t_1,p_1), (t_2,p_2) \in V$, 
if $per(p_1)=per(p_2)=n$ then $[\sigma(p_1)]=[\sigma(p_2)]$.
\item[$(2)$] If the orbit of $p$ for $(s,p)\in V$ does not contain 
the critical point and $per(I(p))=n$, 
then $per(\mu(I(p)))=n/2$.
\end{description}
\end{proposition}

\begin{proof}(1) 
We write $V_{n/2}=\{\,(t,x)\in V\,|\, per(x)=n/2\,\}$. 
Note that $V_{n/2}\subset \{\,t<1\,\}$, 
because by Proposition~\ref{prop:diff-is-minus-one}(1), 
$(f^{n/2}_t)'(x)=-1$ for $(t,x)\in V_{n/2}$, and 
for $t=1$, any periodic point must be hyperbolic. 
Since $V_{n/2}$ is a bounded closed set, 
it is compact. 

Let $V_{n/2}=K_1 \cup \cdots \cup K_v$ be the decomposition of $V_{n/2}$ 
into the sets of the same shuffle. 
Namely, the shuffles of points in each $K_i$ are the same, 
and for $K_i \ne K_j$, the shuffles of points in $K_i$ and $K_j$ 
are different.
Since there are only a finite number of shuffles of period $n/2$, 
$V_{n/2}=\cup K_i$ is a finite decomposition.
Each $K_i$ is a compact set, 
because by Proposition~\ref{prop:same-shuffle} (1), 
each $K_j$ is open in $V_{n/2}$ and $\cup_{j\ne i} K_j$ is open in $V_{n/2}$. 
Therefore, each $K_i$ is closed in $V_{n/2}$ and since $V_{n/2}$ is compact, 
$K_i$ is compact. 

On the other hand, 
since $V_{n/2}$ is closed in $V$, 
$V-V_{n/2}$ is an open set in $V$. 
Since the minimal period of any point in $V-V_{n/2}$ is $n$ 
and there are only a finite number of shuffles of length $n$, 
$V-V_{n/2}$ is divided into a finite number of disjoint subsets 
$\{E_j\}$ of the same shuffle class. 
By Proposition~\ref{prop:same-shuffle} (1), 
each $E_j$ is an open set in $V$. 
We claim the following lemma.

\begin{lemma}\label{lemma:KiEj}
For any $K_i$, there exists a unique $E_j$ such that 
$K_i \cap \overline{E_j}\ne \emptyset$. 
\end{lemma}

If this lemma is true, then for each $K_i$ there is a unique $E_{t_i}$ 
such that $K_i \cap \overline{E_{t_i}}\ne \emptyset$. 
We define $H_i=K_i\cup E_{t_i}$ for all $1\leq i \leq v$. 
And let $H_0=E_{j_1}\cup \cdots \cup E_{j_p}$ be the union of 
$E_{j_k}$'s which do not have a $K_i$ such that 
$K_i \cap \overline{E_{j_k}}\ne \emptyset$. 
Then $V=H_0\cup H_1 \cup \cdots \cup H_v$. 
We claim the following lemma. 
We give the proofs of those lemmas at the end of this section.

\begin{lemma}\label{lemma:separating}
For any $0\leq i \leq v$ and $0\leq j \leq v$, if $i\ne j$ 
then $\overline{H_i} \cap H_j=\emptyset$.
\end{lemma}

Since $V$ is connected, 
this lemma says that only one $H_i$ exists and $V=H_i$. 
Note that $V$ is not $H_0$ because $V_{n/2}$ is not empty. 
$V=H_i=K_i\cup E_{t_i}$ means that $[\sigma(p_1)]=[\sigma(p_2)]$ 
for any $(t_1,p_1), (t_2,p_2) \in V$ of $per(p_1)=per(p_2)=n$, 
and this proves (1). 

(2) As mentioned above, 
$(f^{n/2}_t)'(x)=-1$ at any $(t,x)\in V_{n/2}$. 
Since $V_{n/2}$ is compact, 
there is a neighborhood of $U$ of $V_{n/2}$ in $\BR^2$ such that 
the orbit of any point in $V \cap U$ does not contain the critical point.
Since $V$ is connected, there exists a $(t,x)\in (V-V_{n/2})\cap U$ 
sufficiently close to $V_{n/2}$ such that $per(x)=n$ and $per(I(x))=n/2$. 
Since $per(I(p))=n$, $per(p)=n$. 
It follows from (1) that $[\sigma(p)]=[\sigma(x)]$. 
Thus either $\sigma(p)=\sigma(x)$ or 
$\sigma(p)=\gamma(\sigma(x))$. 
If $\sigma(p)=\sigma(x)$, 
then by Proposition~\ref{prop:shuffle},
either $I(p)=I(x)$, or $per(I(x))=n$ and $I(p)=\mu(I(x))$, 
or $per(I(p))=n$ and $I(x)=\mu(I(p))$. 
Since $per(I(x))=n/2$ and $per(I(p))=n$, 
we have $I(x)=\mu(I(p))$. 
Therefore $per(\mu(I(p)))=n/2$ in this case. 
If $\sigma(p)=\gamma(\sigma(x))=\sigma(f_t^{n/2}(x))$, 
then by Proposition~\ref{prop:shuffle} again, 
either $I(p)=I(f_t^{n/2}(x))$, or $per(I(f_t^{n/2}(x)))=n$ 
and $I(p)=\mu(I(f_t^{n/2}(x)))$, 
or $per(I(p))=n$ and $I(f_t^{n/2}(x))=\mu(I(p))$. 
Since $per(I(f_t^{n/2}(x)))=n/2$ and $per(I(p))=n$, 
we have $I(f_t^{n/2}(x))=\mu(I(p))$. 
Therefore $per(\mu(I(p)))=n/2$ in this case too.
\end{proof}

\begin{proof}[Proof of Lemma~\ref{lemma:KiEj}]
First of all, we prove that for any $K_i$, 
there exists an $E_j$ such that $K_i \cap \overline{E_j}\ne \emptyset$. 

Suppose that there exists a $K_i$ such that 
it does not have an intersection with any $\overline{E_j}$. 
Let $K=K_{i_1}\cup \cdots \cup K_{i_\ell}$ be the union of all 
such $K_i$'s, and $K'=K_{j_1}\cup \cdots \cup K_{j_m}$ the union of all 
$K_i$'s which have an intersection with some $\overline{E_j}$. 
Then, $\left(\cup \overline{E_j}\right) \cup K'$ and $K$ are disjoint closed sets, 
and the union is $V$. 
That contradicts the connectivity of $V$. 

Secondly, we show that such $E_j$ is unique. 
If there is a $(t,x)\in K_i \cap \overline{E_j}$, 
then a sequence of periodic points of period $n$ in $E_j$ 
converges to $(t,x) \in K_i$. 
In such case, as mentioned in the proof of 
Proposition~\ref{prop:period-changes}, 
the way of convergence is joining the points of the orbit 
two by two from the smallest one. 
Since $(f^{n/2}_t)'(x)=-1$ for $(t,x)\in V_{n/2}$, 
there is not the critical point in the orbit of $x$. 
Therefore, the itinerary of periodic point of $E_j$ converging 
to $x$ is the same as $I(x)$ if it is sufficiently close to $x$. 
Moreover, $f_t$ is a local homeomorphism on a neighborhood of each point of 
the orbit of $x$, and whether $f_t$ and nearby $f_s$ are 
orientation-preserving or not on them is defined uniquely by $I(x)$. 
That means that the shuffle class of $E_j$ is uniquely defined 
by $K_i$. 
\end{proof}

\begin{proof}[Proof of Lemma~\ref{lemma:separating}]
Suppose that $\overline{H_i}\cap H_j \ne \emptyset$. 
This means that there is a sequence of points $x_\ell$ in $H_i$ 
converging to a point $p\in H_j$. 

$p\in K_j$ or $p\in E_{t_j}$. 
By taking a subsequence, we can assume that 
$x_\ell \in K_i$ or $x_\ell \in E_{t_i}$ for all $\ell$. 
(If $H_i=H_0$ or $H_j=H_0$, then $K_i=\emptyset$ or $K_j=\emptyset$ 
respectively.)
However since $K_i$ is compact, points in $K_i$ do not converge to 
a point in $E_{t_j}$ or $K_j$. 
Therefore there are only the following two cases. 
\begin{itemize}
\setlength{\itemsep}{0cm} 
\setlength{\parskip}{0pt}
\item[(i)] $x_\ell \in E_{t_i}$ and $p\in E_{t_j}$. 
\item[(ii)] $x_\ell \in E_{t_i}$ and $p\in K_j$. 
\end{itemize}
By Proposition~\ref{prop:same-shuffle} (1), 
(i) does not hold. 
(ii) means $\overline{E_{t_i}} \cap K_j \ne \emptyset$ and 
that contradicts Lemma~\ref{lemma:KiEj}. 
\end{proof}

\section{Proof of Theorem~\ref{main-theorem}}\label{sec:proof-main-theorem}

\begin{proof}
{\noindent\bf (i) } 
First, we shall prove that if $p, q\in G_n^{-1}(0)\cap\{t=1\}$ ($p\ne q$) 
are on the same periodic point component and $per(p)=per(q)=n$, 
then $I(p)=\nu(I(q))$.

Let $V$ be the connected component of $G_n^{-1}(0)$ containing $p$ and $q$. 
Write $k=\min\{\, per(x)\,|\,(t,x)\in V\,\}$. 
Then either $k=n$ or $k<n$. 

\ 

{\noindent\bf Case 1. } $k=n$. 

By Proposition~\ref{prop:same-shuffle} (2), 
the shuffle $\sigma(x)$ is the same for any $(t,x)\in V$. 
The orbits of $p$ and $q$ do not contain the critical point of $f_1$, 
because $f_1$ is a horseshoe. 
Therefore, by Proposition~\ref{prop:shuffle}, 
either $I(p)=I(q)$ or $I(p)=\mu(I(q))$ or $I(q)=\mu(I(p))$. 
If $I(p)=I(q)$, then $p$ and $q$ must be identical because 
$f_1$ is a horseshoe again. 
We have $I(p)=\mu(I(q))$ or $I(q)=\mu(I(p))$. 
Since $f_1$ is a horseshoe, 
$per(x)$ and $per(I(x))$ must be identical 
for any periodic point $x$. 
Therefore, by Proposition~\ref{prop:mu-is-involutive}, 
\lq\lq $I(p)=\mu(I(q))$\rq\rq\  and  
\lq\lq $I(q)=\mu(I(p))$\rq\rq\  are equivalent. 
By the definition of $\nu$, we get $I(p)=\nu(I(q))$. 

\ 

{\noindent\bf Case 2. } $k<n$. 

In this case, by Proposition~\ref{prop:k-is-the-half-of-n}, 
$n$ is even and $k=n/2$. 
Then by Proposition~\ref{prop:half-shift} (1), 
$[\sigma(p)]=[\sigma(q)]$. 

If $\sigma(p)=\sigma(q)$, then 
since $f_1$ is a horseshoe, 
there is no critical point in the orbits of $p$ and $q$. 
By Proposition~\ref{prop:shuffle}, either 
$I(p)=I(q)$, or $per(I(q))=n$ and $I(p)=\mu(I(q))$, 
or $per(I(p))=n$ and $I(q)=\mu(I(p))$. 
However all those cases are impossible. 
Because first of all $I(p) \ne I(q)$. 
Secondly, by Proposition~\ref{prop:half-shift} (2), 
we have $per(\mu(I(p)))=n/2$ and $per(\mu(I(q)))=n/2$, 
because $per(I(p))=n$ and $per(I(q))=n$. 
Thus, both $I(p)=\mu(I(q))$ and $I(q)=\mu(I(p))$ are impossible.

Hence we have only 
the case $\sigma(p)=\gamma(\sigma(q))=\sigma\left(f_1^{n/2}(q)\right)$. 
By Proposition~\ref{prop:shuffle}, either 
$I(p)=I\left(f_1^{n/2}(q)\right)$, 
or $per\left( I \left( f_1^{n/2}(q)\right)\right)=n$ 
and $I(p)=\mu\left(I \left(f_1^{n/2}(q)\right)\right)$, 
or $per(I(p))=n$ and $I\left(f_1^{n/2}(q)\right)=\mu(I(p))$. 
But the second and the third cases do not hold, 
because by Proposition~\ref{prop:half-shift} (2) again, 
we have 
$per \left( \mu\left(I \left(f_1^{n/2}(q)\right)\right)\right)=n/2$ 
and $per \left( \mu(I(p)) \right)=n/2$ because 
$per \left( I \left(f_1^{n/2}(q)\right)\right)=n$  
and $per (I(p))=n$. 
Therefore, 
we have only the case $I(p)=I\left(f_1^{n/2}(q)\right)$. 
Since $f_1$ is a horseshoe, we have  $p=f_1^{n/2}(q)$. 
This means $I(p)=\nu(I(q))$ and $I(q)=\nu(I(p))$ by the definition of $\nu$. 

\ 

{\noindent\bf (ii) } 
 Suppose that $I(p)=\nu(I(q))$. 
Let $V$ be the periodic point component containing $(1,q)$. 
By Proposition~\ref{prop:two-points}, 
$V\cap \{t=1\}$ must have at least two points. 
Let $(1,q')$ be one of other points. 
Note that $per(q')=n$ because of the same reason 
as mentioned in the first part of the proof of 
Proposition~\ref{prop:two-points}. 
Since $(1,q)$ and $(1,q')$ are on the same periodic point component $V$, 
by the above argument, $I(q')=\nu(I(q))=I(p)$. 
Then $q'$ must be $p$, 
because $f_1$ is a horseshoe and 
there is only one point whose itinerary is $I(p)$. 
\end{proof}

\section{The proof of Proposition~\ref{prop:connecting-curve}}\label{section-proof-connecting-curve}

Recall some definitions of dimensions. 
For a topological space $X$, the Lebesgue covering dimension $\dim X$ of $X$ is 
less than or equal to $n$ if each finite open cover of $X$ has 
a refinement $\mathcal{V}$ such that no point is included in 
more than $n + 1$ elements of $\mathcal{V}$. 
A small inductive dimension ${\rm ind}\,X$ of a topological space $X$ is 
defined as follows: ${\rm ind}(\emptyset) = -1$. 
By induction, ${\rm ind}\, X \leq n$ if for any point $x \in X$ and 
any open neighborhood $U$ of $x$, there is an 
open neighborhood $V$ of $x$ with $\overline{V} \subseteq U$ such 
that ${\rm ind}( \partial V ) \leq n -1$, 
where the boundary $\partial V$ of $V$ is 
defined by $\partial V := \overline{V} - \mathrm{int}V$. 
A large inductive dimension ${\rm Ind}\, X$ of a topological space $X$ is 
defined as follows: ${\rm Ind}(\emptyset) = -1$. 
By induction, ${\rm Ind}\, X \leq n$ if for any open subset $U$ of $X$ and 
for any closed subset $F \subseteq U$, there is 
an open neighborhood $V$ of $F$ with $\overline{V} \subseteq U$ such 
that ${\rm Ind}( \partial V )\leq n-1$.
By dimension, we mean the small inductive dimension. 
Recall that Urysohn's theorem 
says $\dim X =  {\rm ind}\, X = {\rm Ind}\, X$ for a normal space $X$. 
Note that a metrizable space is normal and so these three dimensions correspond to each other for a metrizable space. 
A compact metrizable space $X$ whose inductive dimension is $n > 0$ is an $n$-dimensional Cantor-manifold if the complement $X - L$ for any closed subset $L$ of $X$ whose inductive dimension is less than $n - 1$ is connected.  
It's known that a compact connected manifold is a Cantor-manifold \cite{HM}\cite{T}. 

By a decomposition, we mean a family $\mathcal{F} $ of pairwise disjoint nonempty subsets of a set $X$ such that $X = \sqcup \mathcal{F}$, where $\sqcup$ denotes a disjoint union.  
Let $\mathcal{F} $ be a decomposition of a topological space $X$. 
For any $x \in X$, denote by $\F(x)$ the element of $\mathcal{F}$ containing $x$. 
For a subset $V$ of $X$, write the saturation $\F(V)  := \bigcup_{x \in V} {\F}(x)$. 
A subset $V \subseteq X$ is saturated if $V = \F(V)$. 
The decomposition $\mathcal{F}$ of a topological space $X$ is upper semicontinuous if each element of $\mathcal{F} $ is both closed and compact and for any $L \in \mathcal{F} $ and for any open neighborhood $U \subset X$ of $L$ there is a saturated neighborhood of $L$ contained in $U$. 
Epstein has shown the following equivalence. 
\begin{lemma}[Remark after Theorem 4.1 \cite{E}]\label{Epstein}
The following are equivalent for a decomposition $\F$ into connected compact elements 
of a locally compact Hausdorff space $X$:  \\
(1) 
$\mathcal{F}$ is upper semicontinuous.    
\\
(2) 
The quotient space $X/\mathcal{F}$ is Hausdorff.  
\\
(3) 
The canonical projection $p: X \to X/\mathcal{F}$ is closed $(\mathrm{i.e.}$ the saturation of a closed subsets is closed$)$. 
\end{lemma}
By a continuum, we mean a nonempty compact connected metrizable space. 
A subset $C$ in a topological space $X$ is separating if the complement $X - C$ is disconnected. 
Define a filling $\mathrm{Fill}_{\R^2}(W) \subseteq \R^2$ of a continuum $W \subseteq \R^2$ as follows:  
$ p \in \mathrm{Fill}_{\R^2}(W)$ if either $p \in W$ or the connected component of $\R^2 - W$ containing $p$ is an open disk whose boundary is contained in $W$. 
Here an open disk means a nonempty simply connected open subset in a plane or a sphere. 

Recall that the boundary of an open disk in a plane can be disconnected in general. 
On the other hand, boundedness implies the following observation for the connectivity of the boundaries of a bounded open disk in a plane and an open disk in a sphere. 

\begin{lemma}\label{connectivity_boundary}
The boundaries of a bounded open disk in a plane and an open disk in a sphere are connected. 
\end{lemma}

\begin{proof}
Let $D$ be either a bounded open disk in a plane $\R^2$ or an open disk in $\mathbb{S}^2$ and $p : D_0 \to D$ a homeomorphism from the unit open disk $D_0$ in a plane. 
Then the boundary $\partial D$ is compact. 
Suppose that the boundary $\partial D$ is disconnected.
There are disjoint nonempty closed subsets $A$ and $B$ whose union is $\partial D$.
Put $d := \min \{ d(a,b) \mid a \in A, b \in B \} > 0$.
Then $U_A := B_{d/2}(A)$ and $U_B := B_{d/2}(B)$ are open neighborhoods of $A$ and $B$ respectively such that $U_A \cap D \neq \emptyset$ and $U_B \cap D \neq \emptyset$. 
Denote by $C(\theta)$ the image of a curve $p(\{ (r \cos \theta , r \sin \theta) \mid r \in [0, 1) \})$ for any $\theta \in [0, 2\pi)$.  
For any $\theta \in [0, 2\pi)$, since the curve $C(\theta)$ in the open disk $D$ is closed as a subspace of $D$ (i.e. $\overline{C(\theta)} \cap D = C(\theta)$), the difference $\overline{C(\theta)} - C(\theta)$ is contained in $\partial D$. 
Since any curve from a point in $D$ to a point in the boundary $\partial D$ intersects $U_A \sqcup U_B$, compactness of $\partial D$ implies that the preimage $p^{-1}((U_A \sqcup U_B) \cap D)$ contains an annulus $\A := \{ (r \cos t , r \sin t) \mid t \in \R, r \in (s, 1) \}$ for some $s \in [0,1)$. 
This implies that $p^{-1}((U_A \cap D) \cap \A$ and $p^{-1}((U_B \cap D) \cap \A$ form an open covering of $\A$ and so that $\A$ is disconnected. 
That contradicts that $\A$ is annular.
Thus $\partial D$ is connected.
\end{proof}
Notice that the previous lemma is true for not only two dimensional open disk but also $n$-dimensional open ball. 
The previous lemma implies the following observation. 

\begin{corollary}\label{connectivity_boundary02}
The boundary of an unbounded open disk in a plane is unbounded unless the disk is $\R^2$. 
\end{corollary}

\begin{proof}
Let $\mathbb{S}^2$ be the one-point compactification of $\R^2$, $\infty$ the point at infinity (i.e. $\{ \infty \} = \mathbb{S}^2 - \R^2$), and $D$ an unbounded open disk which is a proper subset of $\R^2$. 
Unboundedness of $D$ implies that the boundary $\partial_{\mathbb{S}^2} D$ of $D$ in $\mathbb{S}^2$ contains $\infty$. 
Since $D \subsetneq \R^2$, the boundary $\partial_{\mathbb{S}^2} D$ contains a point in $\R^2$. 
Lemma \ref{connectivity_boundary} implies that the boundary $\partial_{\mathbb{S}^2} D$ is connected and so  the boundary $\partial_{\R^2} D = \partial_{\mathbb{S}^2} D - \{ \infty \}$ of $D$ in $\R^2$ is unbounded. 
\end{proof}
We show the following property of the complement of a filling. 

\begin{lemma}\label{lem:complement}
The complement of a continuum $W$ on $\R^2$ consists of one unbounded open annulus and bounded open disks, and the complement of the filling $\mathrm{Fill}_{\R^2}(W)$ is an unbounded open annulus on $\R^2$. 
In other words, the complement of $W$ in the one-point compactification $\mathbb{S}^2$ of $\R^2$ consists of open disks, and the complement of the filling $\mathrm{Fill}_{\R^2}(W)$ is an open disk on $\mathbb{S}^2$ containing the point at infinity. 
\end{lemma}

\begin{proof}
Since $W$ is bounded and closed, the complement $\R^2 - W$ is the union of bounded open disks with or without punctures and exactly one unbounded connected component $\mathbb{A}$. 
We show that each connected component of $\R^2 - W$ which is bounded is simply connected. 
Indeed, assume that there is a connected component $D \subset \R^2$ of the complement $\R^2 - W$ which is bounded but is not an open disk. 
Since each connected component of $\R^2 - W$ is closed in $\R^2 - W$, we have $\partial D \subseteq W$. 
The Riemann mapping theorem states that each nonempty open simply connected proper subset of $\mathbb{C}$ is conformally equivalent to the unit disk, and so the component $D$ is not simply connected. 
Then there is a simple closed curve $\gamma \subset D$ which is not contractible in $D$. 
By the Jordan curve theorem, the complement $\R^2 - \gamma$ consists of an unbounded open annulus $\mathbb{A}_\gamma$ and an open disk $D_\gamma$ each of which intersects a connected component of the boundary $\partial D \subseteq W$. 
Since $\gamma \subset D \subset \R^2 - W$, the disjoint union $\mathbb{A}_\gamma  \sqcup D_\gamma$ of open subsets is an open neighborhood of $W$, which contradicts the connectivity of $W$. 
Similarly, each connected component of $\mathbb{S}^2 - W$ is simply connected and so the connected component $\mathbb{A}$ is an open annulus because $\mathbb{A}$ is an open disk minus the point at infinity.   
%
By Corollary \ref{connectivity_boundary02}, the boundary of any unbounded open disk in $\R^2$ is unbounded unless the disk is $\R^2$. 
Since $W$ is bounded, each open disk whose boundary is contained in $W$ is bounded and so does not intersect $\mathbb{A}$ because the boundary of any bounded disk intersecting $\mathbb{A}$ intersects $\mathbb{A} \subset \R^2 - W$. 
This implies that $\mathbb{A} \cap \mathrm{Fill}_{\R^2}(W) = \emptyset$. 
Since the boundary of any connected component of $\R^2 - W$ is contained in $W$, we have $\mathrm{Fill}_{\R^2}(W) = \R^2 - \mathbb{A}$ and so $\R^2 - \mathrm{Fill}_{\R^2}(W) = \mathbb{A}$. 
In other words, the complement $\R^2 - \mathrm{Fill}_{\R^2}(W)$ is the unbounded open annulus $\A$. 
\end{proof}
We show that the filling of a continuum is a non-separating continuum. 

\begin{lemma}\label{lem:filling}
The filling of a continuum in a plane is a non-separating continuum. 
\end{lemma}

\begin{proof}
Let $W$ be a continuum in $\R^2$. 
By Lemma \ref{lem:complement}, the filling $F_W := \mathrm{Fill}_{\R^2}(W) \subseteq \R^2$ is the complement of an unbounded open annulus, and so is bounded, closed, and non-separating. 
Moreover, the filling $F_W$ is a disjoint union of $W$ and open disks $U_\lambda$ $(\lambda \in \Lambda)$ whose boundaries are contained in $W$. 
Then it suffices to show that $F_W$ is connected. 
Indeed, assume that $F_W$ is disconnected.  
Then there are two disjoint open subsets $U$ and $V$ whose union is a neighborhood of $F_W$ such that $F_W \cap U \neq \emptyset$ and $F_W \cap V \neq \emptyset$. 
Then $F_W \not\subset U$ and $F_W \not\subset V$. 
The connectivity of $W$ implies either $W \subset U$ or $W \subset V$. 
We may assume that $W \subset U$. 
Fix any $\lambda \in \Lambda$. 
Since $\partial U_\lambda \subset W \subset U$, we have $U_\lambda \cap U \neq \emptyset$. 
Since $U$ and $V$ are disjoint nonempty open subsets with $U_\lambda \subset U \sqcup V$, the connectivity of $U_\lambda$ implies that $U_\lambda \subset U$. 
This means that $F_W = W \sqcup \bigsqcup_{\lambda \in \Lambda} U_\lambda \subset U$, which contradicts $F_W \not\subset U$. 
\end{proof}

We state an observation which is a generalization of a part of the Jordan curve theorem. 

\begin{lemma}\label{inclusion_boundary}
A bounded open disk in a plane which contains the boundary of another bounded open disk in the plane contains another open disk. 
\end{lemma}

\begin{proof}
Let $D \subset \R^2$ be a bounded open disk and $D' \subset \R^2$ a bounded open disk with $\partial D \subset D'$. 
By Lemma \ref{connectivity_boundary}, the boundary $\partial D \subset \R^2$ is a bounded closed connected subset and so is a continuum.  
By Lemma \ref{lem:complement}, the complement of $\mathrm{Fill}_{\R^2}(\partial D)$ is an unbounded open annulus $\mathbb{A}$ on $\R^2$ and the complement of the difference $\mathrm{Fill}_{\R^2}(\partial D) - D$ consists of the unbounded open annulus $\mathbb{A}$ and the open disk $D$. 
Moreover, there is an open neighborhood $U$ of $\partial D$ which does not contain $D$ but  is a finite union of open balls of finite radii such that $\overline{U}  \subset D'$ and that each pair of boundaries of such two distinct open balls intersects transversely if it intersects. 
Then each connected component of $ \partial U$ consists of finitely many arcs and so is a simple closed curve. 
This implies that $U$ is a punctured disk whose boundary is a finite union of simple closed curves and that the filling $\mathrm{Fill}_{\R^2}(U)$ of $U$ is a bounded open disk whose boundary is a simple closed curve. 
Since $\partial U$ consists of finitely many simple closed curves contained in the bounded open disk $D'$, the Jordan curve theorem to the open disk $D'$ implies that any simple closed curve which is a connected component of $\partial U \subset D'$ bounds an open disk on $D'$. 
This means that $\mathrm{Fill}_{\R^2}(U) \subset D'$. 
Since $\partial D \subset U$, we have $\overline{D} \subset \mathrm{Fill}_{\R^2}(\partial D) \subset \mathrm{Fill}_{\R^2}(U) \subset D'$. 
\end{proof}

We show that the inclusion relation on the set of fillings of elements of a decomposition is a partial order.

\begin{lemma}\label{inclusion}
For any continua $F \neq F'$ in $\R^2$ with $\mathrm{Fill}_{\R^2}(F) \cap \mathrm{Fill}_{\R^2}(F') \neq \emptyset$, we have either $\mathrm{Fill}_{\R^2}(F) \subset D_{F'} \subset \mathrm{int} (\mathrm{Fill}_{\R^2}(F'))$ or $\mathrm{Fill}_{\R^2}(F') \subset D_{F} \supset \mathrm{Fill}_{\R^2}(F')$, where $D_{F'}$ is some bounded open disk which is a connected component of $\R^2 - F'$ and $D_{F}$ is some bounded open disk which is a connected component of $\R^2 - F$. 
\end{lemma}

\begin{proof}
By Lemma \ref{lem:complement}, the complements of $\mathrm{Fill}_{\R^2}(F)$ and $\mathrm{Fill}_{\R^2}(F')$ (resp. $F$ and $F'$) are unbounded open annuli (resp. unbounded open annuli and bounded open disks) on $\R^2$. 
Since $\mathrm{Fill}_{\R^2}(F) \cap \mathrm{Fill}_{\R^2}(F') \neq \emptyset$, 
fix a point $p \in \mathrm{Fill}_{\R^2}(F) \cap \mathrm{Fill}_{\R^2}(F')$. 
Suppose that $p \in F$. 
Then there is a bounded open disk $D_{F'}$ which is the connected component of $\R^2 - F'$ containing $p$. 
Hence $p \in F \cap D_{F'}$. 
Since $F \cap F' = \emptyset$ and $\partial D_{F'} \subset \partial F' \subset F'$, we have $F \cap \partial D_{F'} = \emptyset$. 
Since the disjoint union $D_{F'} \sqcup (\R^2 - \overline{D_{F'}}) = \R^2 - \partial D_{F'}$ is an open neighborhood of $F$, the connectivity of $F$ implies $F \subset D_{F'}$. 
By Lemma \ref{inclusion_boundary}, we obtain $\mathrm{Fill}_{\R^2}(F) \subset D_{F'} \subset \mathrm{int} (\mathrm{Fill}_{\R^2}(F'))$.
By symmetry, we may assume that $p \notin F \sqcup F'$. 
Then there are bounded open disks $D_F$ and $D_{F'}$ such that $D_F$ (resp. $D_{F'}$) is the connected component of $\R^2 - F$ (resp. $\R^2 - F'$) containing $p$. 
Then $\partial D_F \subset \partial F \subset F$ and $\partial D_{F'} \subset \partial F' \subset F'$.  
Since $F \cap F' = \emptyset$, we have $\partial D_F \cap \partial D_F' = \emptyset$. 
Define a continuous function $f: \R^2 \to \R$ as follows: $f(x) = - \min_{y \in \partial D_{F'}} d(x, y)$ if $x \in D_{F'}$ and $f(x) = \min_{y \in \partial D_{F'}} d(x, y)$ if $x \notin D_{F'}$, where $d$ is the Euclidean distance on $\R^2$.  
We show that $\overline{D_F} \subset D_{F'}$ or $D_F \supset \overline{D_{F'}}$. 
Indeed, since $\partial D_F \cap \partial D_F' = \emptyset$, we have $0 \notin f(\partial D_F)$.  
By Lemma \ref{connectivity_boundary}, the boundary $\partial D_F$ is connected and so we obtain either $f(\partial D_F) \subset \R_{<0}$ or $f(\partial D_F) \subset \R_{>0}$. 
This means either that $\partial D_F \subset D_{F'}$ or $\partial D_F \subset \R^2 - D_{F'}$. 
Suppose that $\partial D_F \subset D_{F'}$. 
By Lemma \ref{inclusion_boundary}, we have $\overline{D_F} \subset D_{F'}$. 
%
Thus we may assume that $\partial D_F \subset \R^2 - D_{F'}$. 
Since $f(p) < 0$, we have $0 \in f(D_F)$ and so $D_F \cap \partial D_{F'} \neq \emptyset$. 
Similarly, define a continuous function $f': \R^2 \to \R$ as follows: $f'(x) = - \min_{y \in \partial D_{F}} d(x, y)$ if $x \in D_{F}$ and $f'(x) = \min_{y \in \partial D_{F}} d(x, y)$ if $x \notin D_{F}$. 
As the same argument, we may assume that $\partial D_{F'} \subset D_{F}$ or $\partial D_{F'} \subset \R^2 - D_{F}$. 
Therefore $\partial D_{F'} \subset D_{F}$. 
Lemma \ref{inclusion_boundary} implies $\overline{D_{F'}} \subset D_{F}$. 
Since either $\overline{D_F} \subset D_{F'}$ or $D_F \supset \overline{D_{F'}}$, by symmetry, we may assume that $\overline{D_F} \subset D_{F'}$.  
Then $F \cap D_{F'} \neq \emptyset$. 
Since $F \cap F' = \emptyset$, we have $F \cap \partial D_{F'} = \emptyset$ and so $0 \notin f(F)$. 
The connectivity of $F$ implies that $f(F) \subset \R_{<0}$. 
This means that $F \subset D_{F'}$ and so $\mathrm{Fill}_{\R^2}(F)  \subset D_{F'} \subset \mathrm{int}(\mathrm{Fill}_{\R^2}(F'))$.
\end{proof}

We recall the following tool.

\begin{lemma}[Moore's theorem (cf. p.3 in \cite{D})]\label{Moore's}
Let $S$ be a plane or a sphere.  
The quotient space $S/\F$ of an upper semicontinuous decomposition $\F$ of $S$ into non-separating continua is homeomorphic to $S$ unless $\F$ is the singleton of the sphere. 
\end{lemma}


Fix a square $D := [0,1] \times [-M, M]$ and 
a continuous function $G: D \to \R$ such that 
$G^{-1}(0) \cap \partial D \subseteq \{ 1 \} \times (-M, M)$ is finite. 
Let $R: [0,1] \times [-M, M] \to [1,2] \times [-M, M]$ be the reflection with respect to $\{1\} \times [-M, M]$ (i.e. $R( x, y ) = ( 2 - x, y)$). 
For a subset $B \subset \R^2$, define the union $\hat{R}(B) := B \cup R(B)$. 
Consider the double $[0,2] \times [-M, M] = \hat{R}([0, 1] \times [-M, M])$ with respect to $\{1\} \times [-M, M]$. 
Extend $G$ to $\hat{R}(D) = R(D) \cup D$ by  $G|_{R(D)} = G \circ R$ (see Figure \ref{figure_4}). 
Note that the statement of Proposition \ref{prop:connecting-curve} for the original continuous map $G$ is equivalent to one of the extended continuous map $G$. 
Therefore we deal with the extended continuous map $G: [0,2] \times [-M, M] \to \R$ to show the statement. 
Denote by $\mathcal{G}_0$ the set of connected components of $G^{-1}(0) \subset (0,2) \times (-M, M)$.

\begin{figure}
\begin{center}
\includegraphics[scale=0.25]{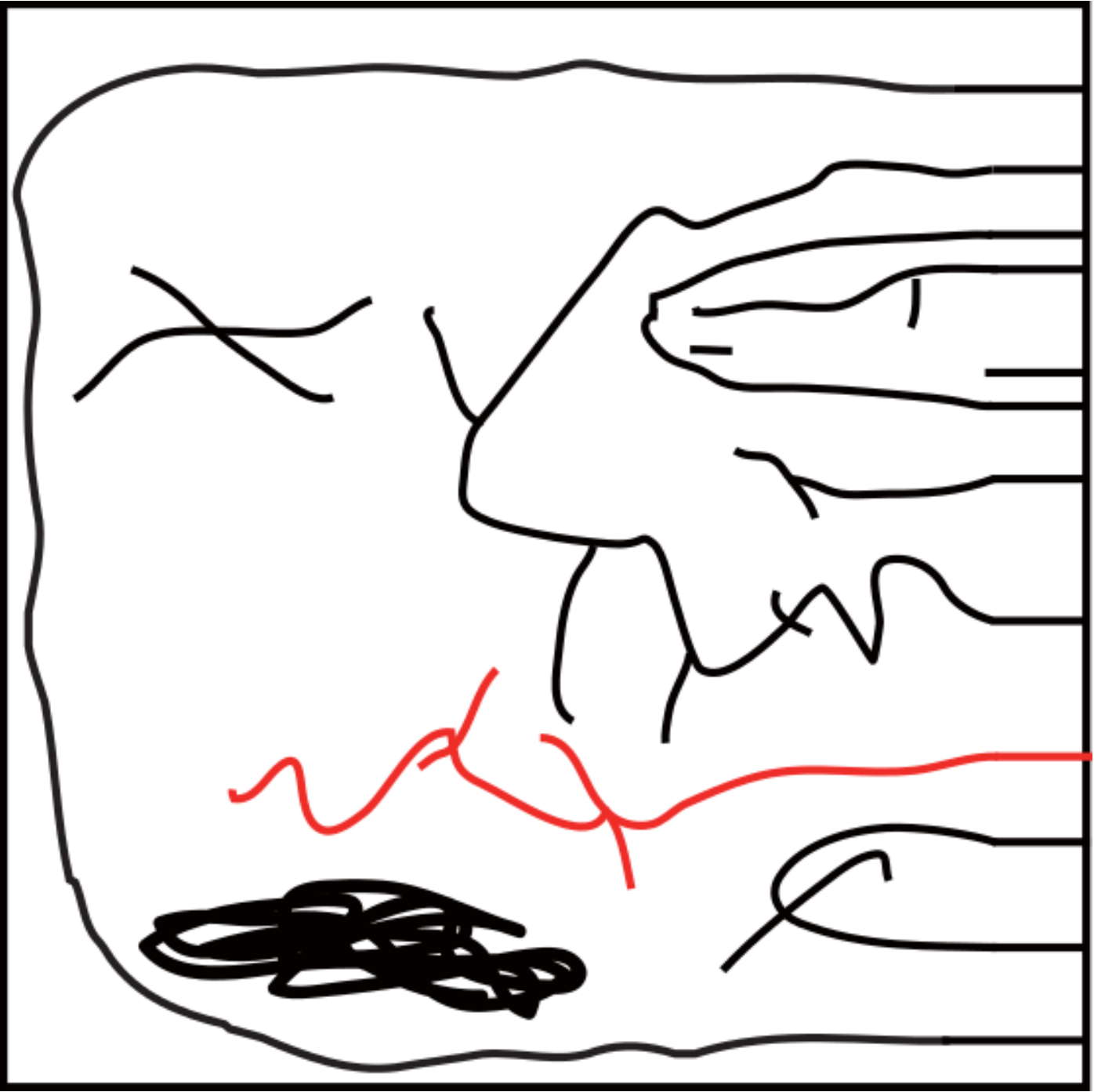}
\includegraphics[scale=0.25]{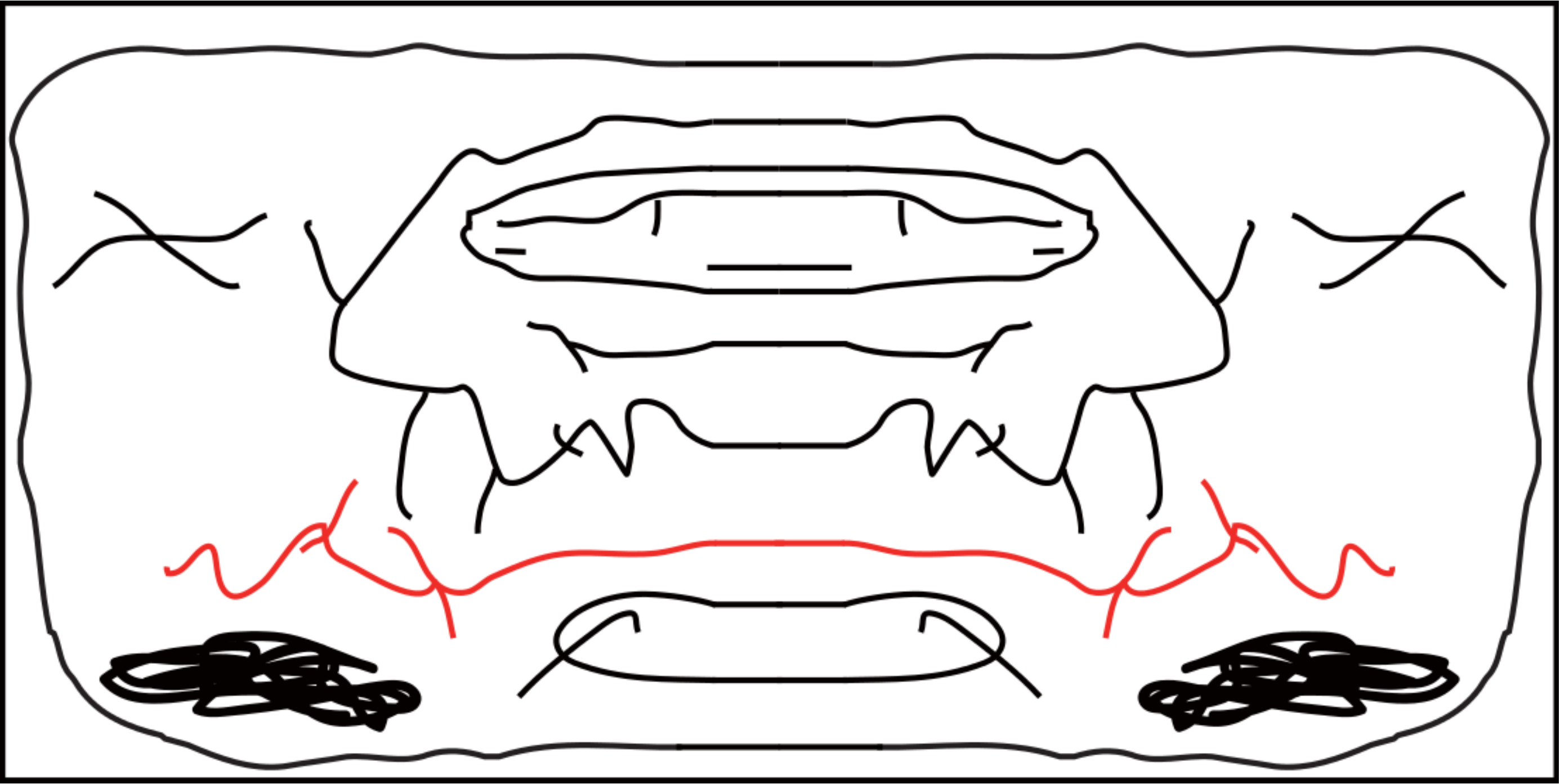}
\end{center} 
\caption{The original domain $[0,1] \times [-M, M]$ and its double $[0,2] \times [-M, M]$ with respect to $\{ 1 \} \times [-M, M]$}
\label{figure_4}
\end{figure}

\begin{lemma}\label{Open-disk}
Let $\mathcal{C}_0 := \{ C \in \mathcal{G}_0 \mid \mathrm{Fill}_{\R^2}(C_0) \subseteq \mathrm{Fill}_{\R^2}(C) \}$ for a continuum $C_0 \in \mathcal{G}_0$. 
Then the family $\mathcal{F}_{C_0} := \{ \mathrm{Fill}_{\R^2}(C) \mid C \in \mathcal{C}_0 \}$  is a totally ordered set with respect to the inclusion relation and has a maximal element. 
\end{lemma}

\begin{proof}
By Lemma \ref{inclusion}, the set $\mathcal{F}_{C_0}$ with the inclusion is a totally ordered set. 
Put $D_0 := \bigcup \mathcal{F}_{C_0} = \bigcup \{ \mathrm{Fill}_{\R^2}(C) \mid C \in \mathcal{C}_0 \} \subset [0,2] \times [-M, M]$. 
Assume that $\mathcal{F}_{C_0}$ has no maximal element. 
By Lemma \ref{inclusion}, there is a family $\mathcal{D}_0 := \{ D_F \}_{F \in \mathcal{F}_{C_0}}$ of bounded open disks with $D_0 = \bigcup \mathcal{D}_0$ which is total ordered with respect to the inclusion relation. 
We show that each closed curve $\gamma$ on $D_0$ is null homotopic in $D_0$. 
Indeed, since $\gamma$ is compact and the family $\mathcal{D}_0$ consists of open disks and is an open covering of $\gamma$ and the totally ordered set with respect to the inclusion relation, there is an element $D \in \mathcal{D}_0$ such that $\gamma \subset D$ and so the curve $\gamma$ is null homotopic. 
Therefore the union $D_0$ is a bounded open disk on $\R^2$. 
By Lemma \ref{connectivity_boundary}, the boundary $\partial D_0$ is connected. 
Moreover we have $\partial D_0 \subseteq \overline{\bigcup \{ \partial \mathrm{Fill}_{\R^2}(C) \mid C \in \mathcal{C}_0 \}} \subseteq \overline{\bigcup \{ \partial C \mid C \in \mathcal{C}_0 \}} \subseteq G^{-1}(0)$. 
Since $\partial D_0 \subseteq G^{-1}(0)$ is connected, there is an element $C_{\infty} \in \mathcal{G}_0$ such that $\partial D_0 \subseteq C_{\infty}$ and so $D_0 \subset \mathrm{Fill}_{\R^2}(C_{\infty})$. 
Since $\mathrm{Fill}_{\R^2}(C_0) \subseteq D_0$, we obtain $\mathrm{Fill}_{\R^2}(C_0) \subseteq \mathrm{Fill}_{\R^2}(C_{\infty})$ and so $C_{\infty} \in \mathcal{C}_0$. 
This means that $\mathrm{Fill}_{\R^2}(C_{\infty})$ is the maximal element of $\mathcal{F}_{C_0}$, which contradicts the assumption. 
\end{proof}

We will show the following statement which is a statement of Proposition \ref{prop:connecting-curve} for the extended continuous map $G$ on the double $[0,2] \times [-M, M]$, and which is equivalent to Proposition \ref{prop:connecting-curve}.

\begin{lemma}\label{existence-path}
If there is a connected component of $G^{-1}(0)$ which intersects exactly once to $\{ 1 \} \times [-M, M]$ at a point $p_0 = (1, y_0)$, then there is a path in $D - G^{-1}(0)$ connecting points $y_-$ and $y_+$ in $\{ 1 \} \times [-M, M]$ which are arbitrary close to $y_0$ with $y_- < y_0 < y_+$. 
\end{lemma}

\begin{proof}
Since $\partial \hat{R}(D) \cap G^{-1}(0) = \emptyset$, 
we may assume that $G$ is constant on $\partial \hat{R}(D) = \partial ([0,2] \times [-M, M])$. 
Collapsing the boundary $\partial \hat{R}(D)$ into a point, denoted by $\infty$, 
the resulting surface is a sphere, denoted by $\mathbb{S}^2$ (see the left figure in Figure \ref{figure_5}). 
Then the induced map $G : \mathbb{S}^2 \to \R$ is a well-defined continuous map. 
Suppose that there is a connected component $L_C$ of $G^{-1}(0)$ which intersects exactly once to $\{ 1 \} \times [-M, M]$ at a point $(1, y_0)$. 
Let $L_1, \ldots , L_k$ be the connected components except $L_C$ of $G^{-1}(0)$ intersecting $\{ 1 \} \times [-M, M]$. 
Write $F_C := \mathrm{Fill}_{\R^2}(L_C)$. 
Then the complement $\D := \mathbb{S}^2 - F_C$ is an open disk. 
Put $F_i :=  \mathrm{Fill}_{\D}(L_i)$ for $i = 1, \ldots , k$. 
Lemma \ref{lem:filling} implies that the fillings $F_C, F_1, \ldots , F_l$ are non-separating  continua. 
By Lemma \ref{inclusion}, we have that either $F_i \cap F_j = \emptyset$, $F_i \subset F_j$, or $F_i \supset F_j$ for any pair $i \neq j$. 
By Lemma \ref{Open-disk}, we may assume that $F_1, \ldots , F_l$ are the maximal elements with respect to the inclusion relation. 

Define a decomposition $\F_0$ on $\mathbb{S}^2$ which consists of connected components of $F_C, F_1, \ldots, F_l$ and singletons of the points of the complement of $F_C \sqcup F_1 \sqcup \dots \sqcup F_l$. In other words,  
$$\F_0 = \{ F_C, F_1, \ldots, F_l \} \sqcup \{ \{ x \} \mid x \in \mathbb{S}^2 - F_C \sqcup F_1 \sqcup \dots \sqcup F_l \}$$ 
Since $F_C, F_1, \ldots, F_l$ are closed and $\mathbb{S}^2$ is normal, the quotient space $\mathbb{S}^2/\mathcal{F}_0$ is Hausdorff. 
Lemma \ref{Epstein} implies that the decomposition $\F_0$ is upper semicontinuous. 
Since each element of $\F_0$ is non-separating, 
applying the Moore's theorem to a decomposition $\F_0$ of $\mathbb{S}^2$, 
the quotient space $\mathbb{S}^2/\F_0$ is a sphere (see the middle figure in Figure \ref{figure_5}). 
This means that there are finitely many connected components of $G^{-1}(0)/\F_0$ intersecting $(\{ 1 \} \times [-M, M])/\F_0$, which are singletons. 
Recall that $\mathcal{G}_0$ is the set of connected components of $G^{-1}(0)$ 
and that $\max P$ is the subset of maximal elements of a partial order set $P$. 
Putting $\mathcal{M} := \max \{ \mathrm{Fill}_{\R^2}(L) \mid L \in \mathcal{G}_0 - \{ F_C, F_1, \ldots, F_l \} \} \sqcup \{ F_C, F_1, \ldots, F_l \}$ with respect to the inclusion relation, 
Lemma \ref{Open-disk} implies that the family $\mathcal{M}$ and the points of the complement $\mathbb{S}^2  - \bigcup \mathcal{M}$ form a decomposition $\F_1$. 
In other words, the decomposition $\F_1$ is defined by 
$$\F_1 :=  \mathcal{M} \sqcup \{ \{ x \} \mid x \in \mathbb{S}^2  - \bigcup \mathcal{M} \}$$
Note that $\mathcal{M} = \max \{ \mathrm{Fill}_{\R^2}(L) \mid L \in \mathcal{G}_0 - \{ F_C, F_1, \ldots, F_l \} \} \sqcup \{ F_C, F_1, \ldots, F_l \} = \max \{ \mathrm{Fill}_{\R^2}(L) \mid L \in \mathcal{G}_0, L \cap \{ 1 \} \times [-M,M] = \emptyset \} \sqcup \{ F_C, F_1, \ldots, F_l \}$. 
Lemma \ref{lem:filling} implies that each element of $\mathcal{M}$ is a non-separating  continuum. 
Since each element of $\F_1$ is closed and $\mathbb{S}^2$ is normal, the quotient space $\mathbb{S}^2/\mathcal{F}_1$ is Hausdorff. 
Lemma \ref{Epstein} implies that the decomposition $\F_1$ is upper semicontinuous. 
Since each element of $\F_1$ is non-separating, 
applying the Moore's theorem to a decomposition $\F_1$, 
the quotient space $\mathbb{S}^2/\F_1$ is a sphere (see the right figure in Figure \ref{figure_5}). 
Since a locally compact Hausdorff space is zero-dimensional if and only if it is totally disconnected, a compact totally disconnected subset $G^{-1}(0)/\F_1$ of a sphere is zero-dimensional. 
Since a sphere is a Cantor manifold, the complement $(\mathbb{S}^2 - G^{-1}(0))/\F_1$ is a connected surface. 
Since a manifold is connected if and only if it is path-connected, the complement $(\mathbb{S}^2 - G^{-1}(0))/\F_1$ is path-connected. 
Since $\F_1(x) = \{ x \}$ for any point $x \in \mathbb{S}^2 - \F_1(G^{-1}(0))$, we have that $\mathbb{S}^2 - (\F_1(G^{-1}(0)) \sqcup \{ \infty \})$ is path-connected. 
Since $\mathbb{S}^2 - (\F_1(G^{-1}(0)) \sqcup \{ \infty \}) = \mathrm{int} \hat{R}(D) - G^{-1}(0)$ is symmetric with respect to $\{ 1 \} \times [-M, M]$, there is a path in $D - G^{-1}(0)$ connecting points $y_-$ and $y_+$ in $\{ 1 \} \times [-M, M]$ which are arbitrary close to $y_0$ with $y_- < y_0 < y_+$. 
\end{proof}

\begin{figure}
\begin{center}
\includegraphics[scale=0.2]{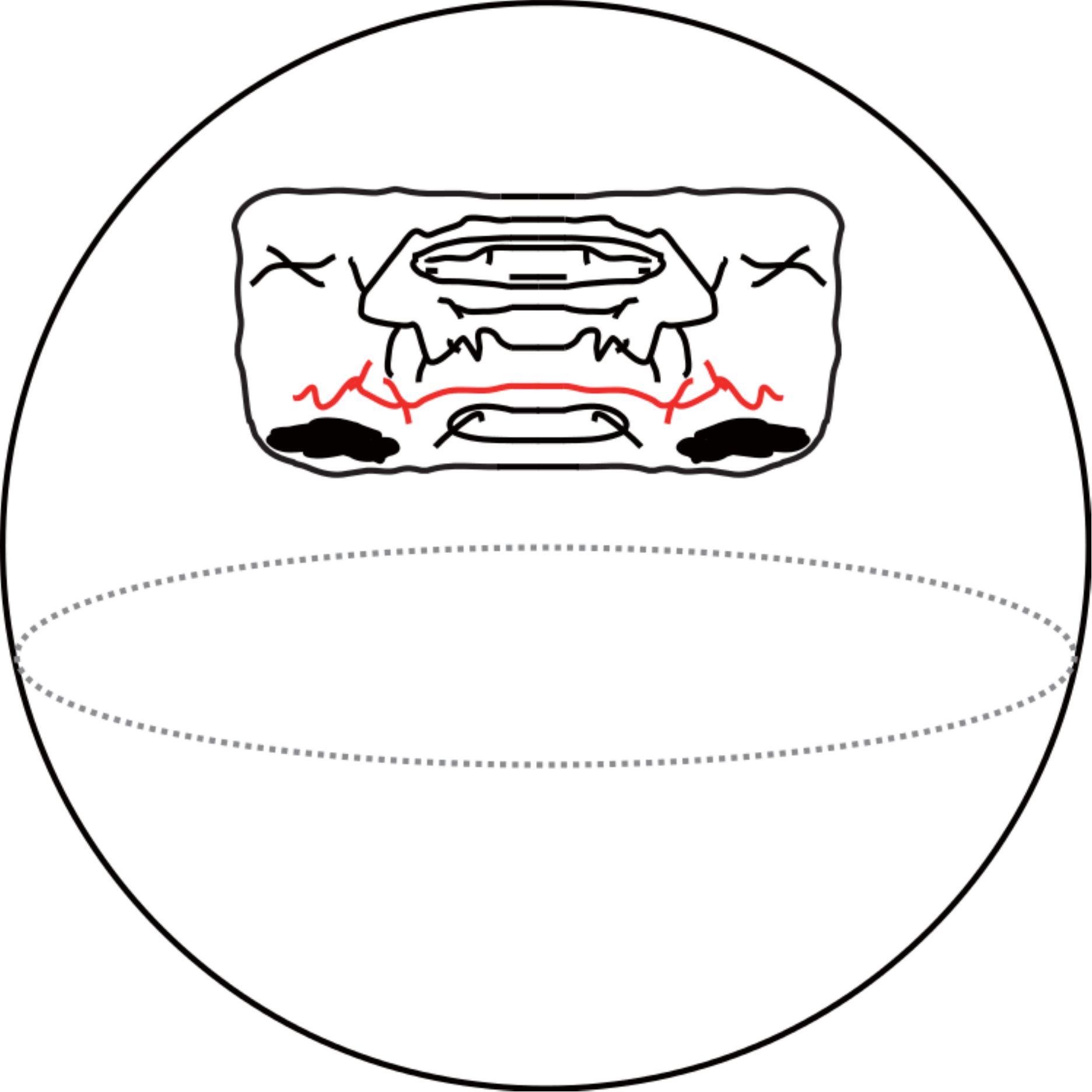}
\includegraphics[scale=0.2]{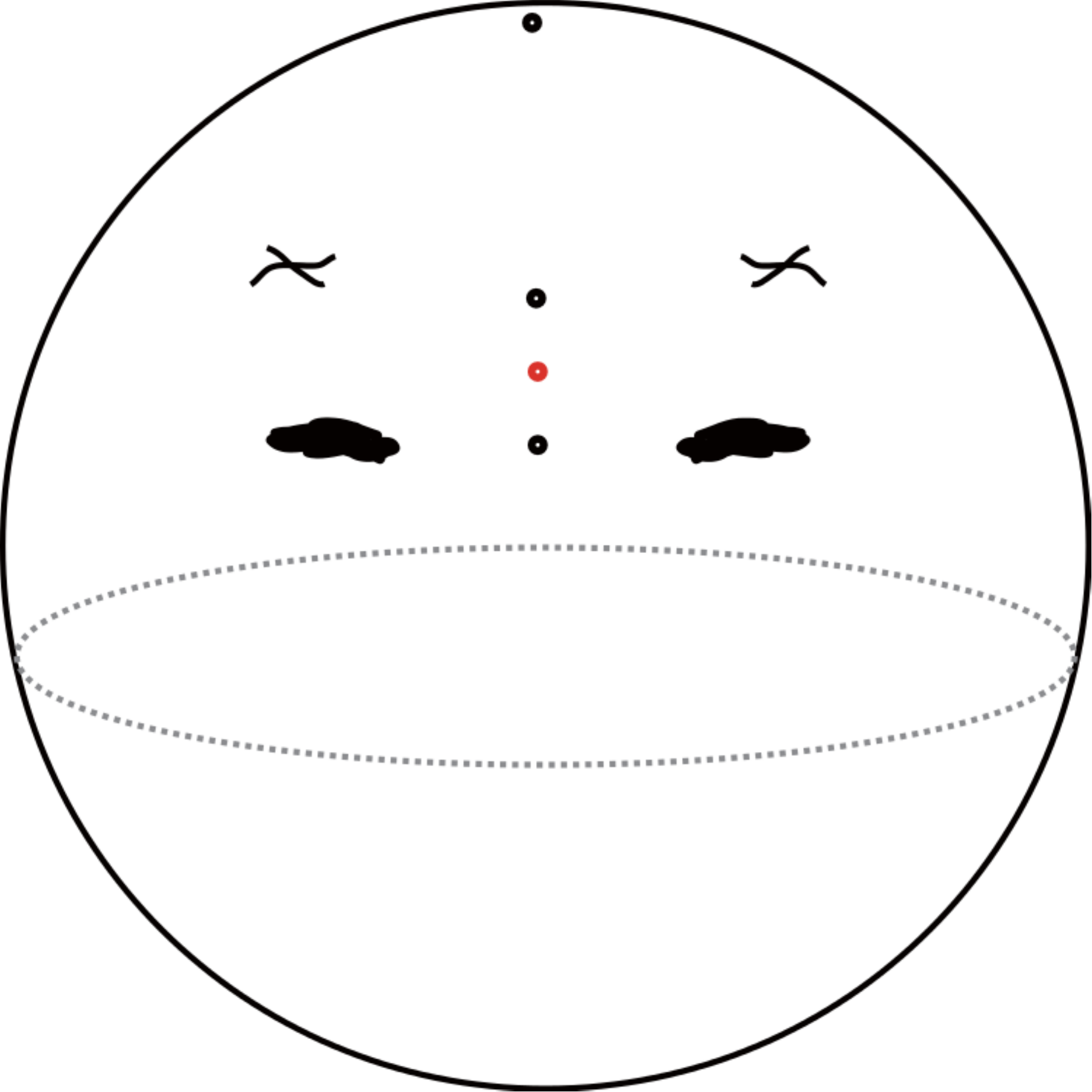}
\includegraphics[scale=0.2]{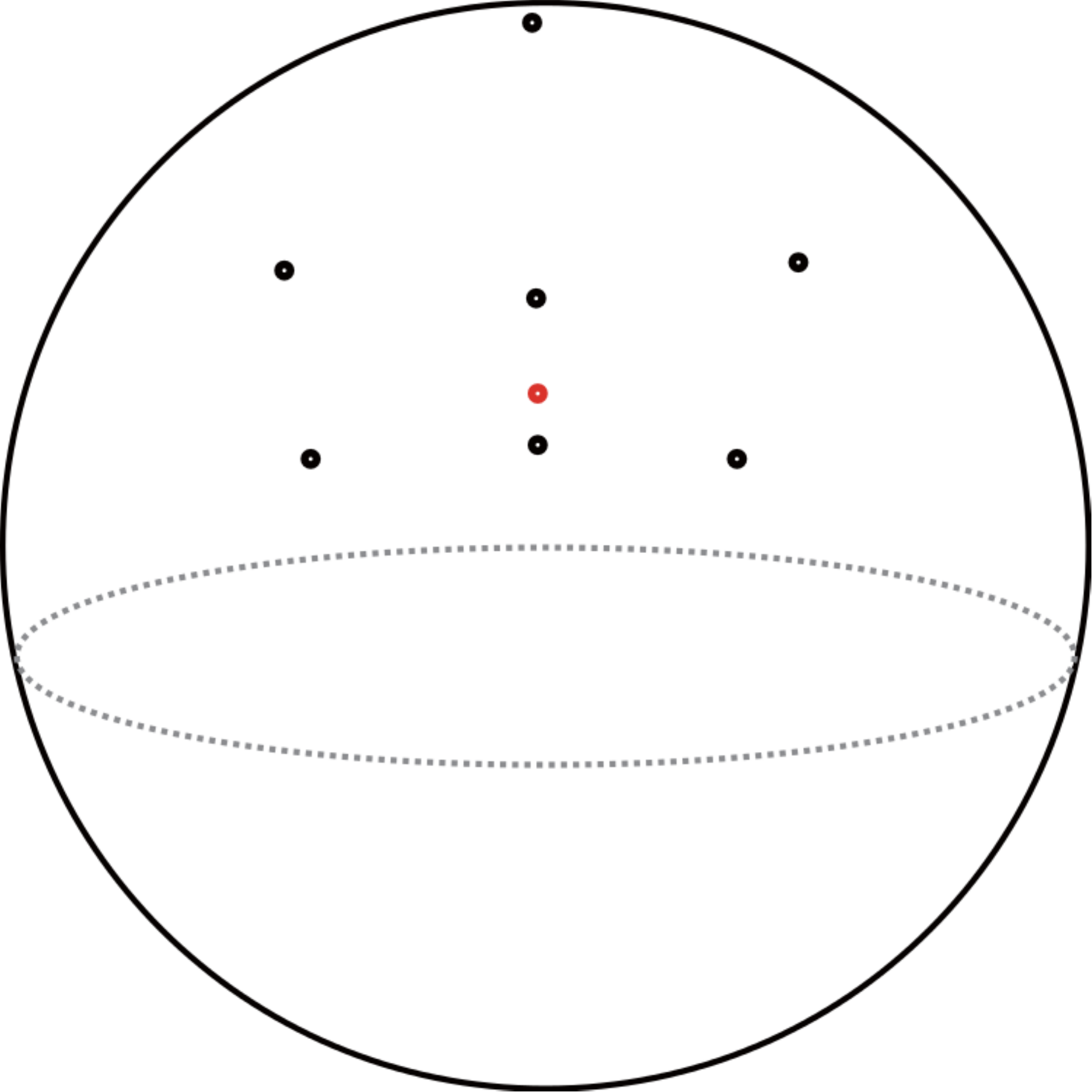}
\end{center} 
\caption{Collapsing processes on the sphere $\mathbb{S}^2 = ([0,2] \times [-M, M])/\partial ([0,2] \times [-M, M])$}
\label{figure_5}
\end{figure}

\section{Appendix}

Roughly speaking, Proposition~\ref{prop:connecting-curve} asserts the existence of separating chord. 
To state this, we define a separating chord as follows: 
Let $D^2$ be the unit disk in $\R^2$. 
An injective continuous curve $I : [0,1] \to D^2$ is a chord 
if $\partial D^2 \cap \mathop{\mathrm{Im}} I =  \{ I(0), I(1) \}$. 
For a compact subset $K \subset D^2$ such that $K \cap \partial D^2$ is finite 
and for a connected component $C$ of $K$ intersecting the boundary $\partial D^2$, 
a chord $I$ is a separating chord from $C$ if $K \cap \mathop{\mathrm{Im}} I = \emptyset$ and 
the complement $D^2 - \mathop{\mathrm{Im}} I$ consists of two connected components 
such that the connected component of $D^2 - \mathop{\mathrm{Im}} I$ containing $C$ contains 
no other connected components of $K$  intersecting $\partial D^2$ (i.e. 
$C \cap \partial D^2 = B \cap K \cap \partial D^2$, 
where $B$ is the connected component of the complement $D^2 - \mathop{\mathrm{Im}} I$ intersecting $C$). 
Finally, we state an existence of separating chord as follows. 

\begin{theorem}
Let $K$ be a compact subset of the unit disk $D^2$ such that $K \cap \partial D^2$ is finite and $C$ a connected component of $K$ intersecting the boundary $\partial D^2$ exactly once. 
Then there is a separating chord from $C$. 
\end{theorem}

\begin{proof}
Define a continuous function $g: D^2 \to \R_{\geq 0}$ by $g(x) := \min \{ d(x, y) \mid y \in K \}$, 
where $d$ is the Euclidean metric. 
Then $K = g^{-1}(0)$. 
Let $h: D = [0,1] \times [-M, M] \to D^2$ be a homeomorphism 
such that $h^{-1}(K) \cap \partial D$ is a finite set on $\{1\}\times (-M,M)$. 
The composition $G := g \circ h: D \to \R$ is a continuous function 
such that $G^{-1}(0) \cap \partial D$ is a finite set on $\{1\}\times (-M,M)$.   
The inverse image $F_C := h^{-1}(C)$ is a connected component of $G^{-1}(0)$ which 
intersects exactly once to $\{ 1 \} \times (-M, M)$ at a point $p_0 = (1, y_0)$. 
Applying Proposition~\ref{prop:connecting-curve}, 
then there is a path $J:[0,1] \to D$ in $D - G^{-1}(0)$ connecting 
points $y_-$ and $y_+$ in $\{ 1 \} \times (-M, M)$ with $y_- < y_0 < y_+$ such 
that $G^{-1}(0) \cap (\{ 1 \} \times (y_-, y_+)) = \{ p_0 \}$. 
Then the composition $h \circ J:[0,1] \to D^2$ is a separating chord from $C$. 
\end{proof}

\

{\bf Acknowledgement}: 
We would like to thank Hiroshi Kokubu for his helpful comments. 
One of the authors submitted the first version of this paper 
to a certain journal a long time ago. 
An anonymous referee pointed out a gap in the proof of 
the main theorem, and made some very useful comments. 
We also would like to thank the referee at that time. 
The second author is partially supported by the JST PRESTO Grant Number JPMJPR16ED and by JSPS Kakenhi Grant Number 20K03583
\ 

%

\end{document}